\documentclass[a4paper,12pt]{article}

\usepackage{amsmath}
\usepackage{amsthm}
\usepackage{enumerate}
\usepackage{amssymb,amsfonts,latexsym,mathtools}
\usepackage{mathrsfs}
\usepackage{bm}
\usepackage{cases}
\usepackage{color}

\newcommand{\Levy}{L\'{e}vy}

\newcommand{\deq}{\stackrel{d}{=}}
\newcommand{\ito}{It\^{o}}
\newcommand{\R}{\mathbb{R}}
\newcommand{\F}{\mathscr{F}}
\newcommand{\B}{\mathscr{B}}
\newcommand{\N}{\mathbb{N}}

\newcommand{\e}{\varepsilon}

\newcommand{\cadlag}{c\`{a}dl\`{a}g}
\renewcommand{\P}{\mathbb{P}}
\usepackage{multirow}
\newcommand{\E}{\mathbb{E}}
\usepackage{url}

\numberwithin{equation}{section}

\makeatletter
\renewcommand\section{\@startsection {section}{1}{\z@}%
{-3.5ex \@plus -1ex \@minus -.2ex}%
{2.3ex \@plus.2ex}%
{\normalfont\large\bf}}
\makeatother

\makeatletter
\renewcommand\subsection{\@startsection {subsection}{1}{\z@}%
{-3.5ex \@plus -1ex \@minus -.2ex}%
{2.3ex \@plus.2ex}%
{\normalfont\normalsize\bf}}
\makeatother

\theoremstyle{plain}
\newtheorem{thm}{Theorem}[section]
\newtheorem{lem}[thm]{Lemma}
\newtheorem{cor}[thm]{Corollary}
\newtheorem{prop}[thm]{Proposition}
\theoremstyle{definition}
\newtheorem{Rem}[thm]{Remark}

\allowdisplaybreaks

\usepackage[dvipdfmx]{hyperref}
\hypersetup{
setpagesize=false,
 bookmarksnumbered=true,
 bookmarksopen=true,
 colorlinks=true,
 linkcolor=blue,
 citecolor=red,
}

\begin{document}
\begin{center}
\Large \textbf{A Tanaka-Type Formula for Compact Sets and Equilibrium Measures of \Levy\ Processes}
\end{center}
\begin{center}
Kohki Iba\footnote{
\begin{tabular}[t]{@{}l@{}}
Affiliation: Graduate School of Science, The University of Osaka, Osaka, Japan.\\
E-mail: \url{kohki.iba@gmail.com}
\end{tabular}
} and V\'{i}ctor M. Rivero\footnote{
\begin{tabular}[t]{@{}l@{}}
Affiliation: Centro de Investigacion en Matematicas A.C., Guanajuato, Mexico.\\
Email: \url{rivero@cimat.mx}
\end{tabular}
}
\end{center}
\begin{abstract}
Tanaka's formula is a classical identity for Brownian motion, and Tsukada \cite{Tukada} extended it to \Levy\ processes not necessarily symmetric. From a potential-theoretic point of view, this formula shows that the invariant function for the process killed upon hitting a singleton can be decomposed into the sum of a martingale part and a local time. In this paper, we generalize this singleton setting and derive a Tanaka-type formula for a compact set $B$. To this end, we introduce the equilibrium measure, defined as the rescaled limit of the $q$-capacity measures, and show that the invariant function for the process killed upon hitting $B$ can be represented as the integral, with respect to the equilibrium measure, of the invariant functions associated with processes killed upon hitting singletons, up to an additive constant called the Robin constant. Moreover, when $B$ is an interval, we obtain explicit representations of the equilibrium measure, the Robin constant, and the martingale part for recurrent stable processes as well as for recurrent spectrally negative \Levy\ processes. Finally, we discuss how an analogous Tanaka-type formula can also be established for transient \Levy\ processes.
\end{abstract}

%%%%%%%%%%%%%%%%%%%%%%%%%%%%%%%%%%%%%%%%%%%%%%%%%%%%%%%%%%%%%%%%%%%%%%%%%%%%%%%%%%%%%%%%%%%%%%%%%%%%%%%%%%%%%%%%%%%%%%%%%%%%%%%%%%%%%%%%%%%%%%%%%%%%%%%%%%%%%%%%%%%%%%%%%%%%%%%%%%%%%%%%%%%%%%%%%%%%%%%%%%%%%%%%%%%%%%%%%%%%%%%%%%%%%%%%%%%%%%%%%%%%%%%%%%%%%%%%%%%%%%%%%%%%%%%%%%%%%%%%%%%%%%%%%%%%%%%%

\section{Introduction}
\label{S1}
The famous Tanaka's formula is one of the fundamental identities in the theory of Brownian motion and of semimartingales. For a Brownian motion, it expresses the absolute value of Brownian motion, or more generally the distance from a fixed point, as the sum of a local martingale and a local time:
\begin{align*}
  |B_t-x|=|B_0-x|+\int_0^t \mathrm{sgn}(B_s-x)dB_s+L_t^x,\qquad t\geq 0;
\end{align*}
where $(L_t^x, t\geq 0)$ denotes the local time at $x$. This formula reveals a deep connection between stochastic calculus, local times, and potential theory. Indeed, let $\varphi_{x}:\R\mapsto\R^{+}$ be the function defined by 
\begin{align*}
  \varphi_{x}(y):=|y-x|,\qquad y\in\R,
\end{align*}
for an arbitrarily fixed $x$. By the Markov property, it is easily seen that 
\begin{align*}
  \varphi_{x}(B_t)-\varphi_{x}(B_{0}),\qquad  t\geq 0,
\end{align*}
is an additive functional. Moreover, applying the so-called It\^{o}-Meyer formula to the convex function $\varphi_{x}$, it is obtained that the function $\varphi_{x}$ is excessive for the Brownian motion; the process
\begin{align*}
  Z_t= |B_t-x|=\varphi_{x}(B_t),\qquad t\geq 0,
\end{align*}
is a submartingale and hence a semimartingale. Thus Tanaka's formula provides an explicit Doob-Meyer decomposition of $Z$ into an additive local martingale and a continuous additive functional of bounded variation. In the latter, the process of bounded variation is the local time at $x$. For background on the Doob-Meyer decomposition, see Chapter VII in \cite{DMVolV-VII}. Tanaka's formula can also be derived using the theory of stochastic calculus for Markov processes, see e.g.~\cite{fukushima} and \cite{fukushimaetal}. The extensions and applications of Tanaka's formula have been a very active research topic and it is almost impossible to provide an account of their developments given the vast literature about it, so rather than omitting many, we cite no one at this stage. 

Another motivation for studying $(Z_{t}=\varphi_{x}(B_t), t\geq 0)$ lies in the fact that this is a martingale for the Brownian motion killed at the first hitting time of $x.$ This is actually,  the Radon-Nikodym density process relating the law of a Brownian motion killed at $x$ and that of a Brownian motion conditioned to avoid $x$. The latter law is obtained from the former through the technique of Doob's h-transform. 

As it is well known, Brownian motion is an instance of L\'evy processes, i.e. of processes with right-continuous and left limited paths that have independent and stationary increments. The understanding of L\'evy processes conditioned to avoid the negative half-line or a point has been an active research area, see e.g. \cite{CD}, \cite{Pa}; this is due to its connections with excursion theory, path transformations, penalization theory, among other reasons; see, for instance, \cite{Yano2}, \cite{TY}, \cite{Takeda}, and \cite{IY-3}. Yamada \cite{Yamada}, Salminen-Yor \cite{SY} and Tsukada \cite{T2, Tukada} obtained extensions of Tanaka's formula for symmetric stable process, symmetric \Levy\ processes, for recurrent stable processes and \Levy\ processes not necessarily symmetric, respectively. 

Under the assumption that $X$ is a real-valued \Levy\ process with a continuous resolvent density and some further assumptions, Salminen-Yor \cite{SY} and Tsukada \cite{Tukada} established Tanaka's formula. They proved that there is a function $h,$ called the \emph{renormalized zero resolvent}, such that for any arbitrarily fixed $x\in\R,$ the function $y\mapsto h(y-x),$ $y\in\R,$ plays the same role for the \Levy\ process $X,$ as $\varphi_{x}$ does for the Brownian motion. Details on the renormalized zero resolvent $h$, will be recalled in Subsection \ref{S2.2}. Namely, they proved that there exists a local martingale $(M_t^x, t\geq 0)$, such that
\begin{align}
\label{Tanaka-Levy}
  h(X_t-x)=h(X_0-x)+M_t^x+L_t^x, \qquad t\geq 0;
\end{align}
where $(L_t^x, t\geq 0)$ denotes the local time at $x$ for $X$. Moreover, Yano \cite{Yano2}, Pant\'{i} \cite{Pa}, and Takeda–Yano \cite{TY} proved that the function $h(\cdot-x)$ is an invariant function for $X$ killed at its first hitting time of $x.$ Those authors described the law of the process $X$ conditioned to avoid $x.$ 

D\"{o}ring-Kyprianou-Weissmann \cite{DKW} described the law of a stable L\'evy process conditioned to avoid a closed bounded interval. Iba \cite{Iba} extended the latter and described the law of a \Levy\ process conditioned to avoid a bounded closed set $B$. In their study, the authors described a function $\varphi_B,$ which is invariant for $X$ killed at its first hitting time of the set $B.$ The associated Doob $h$-transform corresponds to that of a L\'evy process conditioned to avoid $B.$   

Our aim is to complete the program described above using instead of a singleton a closed bounded set. To be more precise, in this paper, we will establish the analogue of Tanaka's formula~\eqref{Tanaka-Levy} with $\varphi_{B}$ in place of $h(\cdot-x)$, and the role of the local time at $x$, $L^{x},$ will be played by $L^{B}=(L^{B}_{t}, t\geq0),$ the local time at $B,$ that is, a non-decreasing continuous additive functional that grows only at the times when $X$ visits $B.$ For the best of our knowledge this is the first instance of a Tanaka-like formula where the involved function is invariant for a L\'evy process killed at the first hitting time of a compact set. 

To give a precise statement, we will make a digression to provide further preliminaries and notations.
\subsection{\Levy\ Processes, Resolvent Density, and Local Time}
We will denote by $((X_t, t\ge 0), \{\P_x, x\in \R\})$ a one-dimensional \Levy\ process such that, under $\P_x$, $X_0=x$ a.s. for each $x\in \R$. For a Borel measurable set $A\subset \R$, let $T_A$ be the first hitting time of $A$ for $X$, that is,
\begin{align*}
  T_A:=\inf\{t\ge 0;\ X_t\in A\}.
\end{align*}
For $\lambda\in \R$, we denote the characteristic exponent of $X$ by $\Psi(\lambda)$, that is, $\Psi(\lambda)$ satisfies
\begin{align*}
  \E_0[e^{i\lambda X_t}]=e^{-t\Psi(\lambda)}\qquad \text{for}\ t\ge 0.
\end{align*}
The process $X$ is a recurrent process whenever
\begin{align*}
  \limsup_{q\to0}\int_{B_{r}}\Re\left(\frac{1}{q+\Psi(\lambda)}\right)d\lambda=\infty,\qquad \forall r>0;
\end{align*}
where $B_{r}$ stands for the ball of radius $r$ around zero, and $\Re(z)$ stands for the real part of the complex number $z.$

Throughout this paper, except for Section \ref{App}, we assume that $(X,\P_0)$ is recurrent and that the following condition holds:
\begin{align*}
\label{ConditionA}
  \int_0^\infty \left|\frac{1}{q+\Psi(\lambda)}\right|d\lambda<\infty,\qquad \text{for}\ q>0. \tag{\bf{A}}
\end{align*}
For instance, this condition is satisfied by recurrent stable processes of index $1<\alpha\le 2$.  

Under the condition (\ref{ConditionA}) only, it is known that $X$ has a bounded continuous resolvent density $r_q$, which satisfies
\begin{align*}
  \int f(y)r_q(y-x)dy=\E_x\left[\int_0^\infty e^{-qt}f(X_t)dt\right],\qquad \text{for all}\ q>0\ \text{and}\ \text{for all}\ x\in\R;
\end{align*}
and for any measurable and bounded test function $f:\R\mapsto\R,$ see, e.g., Theorems II.16 and II.19 of \cite{Ber}. 
The resolvent density can be expressed in terms of the characteristic exponent as 
\begin{align}
\label{resol-den}
  r_q(x)=\frac{1}{\pi}\int_0^\infty \Re \left(\frac{e^{-i\lambda x}}{q+\Psi(\lambda)}\right)d\lambda,\qquad x\in\R,\ q>0;
\end{align}
see, e.g., Lemma 2 of \cite{Wink} and Corollary 15.1 of \cite{Tukada}. It is worth noticing that the process is recurrent if and only if $r_q(0)\to \infty$ as $q\to 0$, see, e.g., Theorem I.17 of \cite{Ber} and Theorem 37.5 of \cite{Sato}. Moreover, there is a well-known formula relating the first hitting time of a point $a\in \R$ to the resolvent density:
\begin{align}
\label{HP-Lap}
  \E_x[e^{-qT_{\{a\}}}]=\frac{r_q(a-x)}{r_q(0)},\qquad x\in\R;
\end{align}
see, e.g., Corollary II.18 of \cite{Ber}.

Condition (\ref{ConditionA}) implies the existence of a local time at $a\in \R$, which we denote by $(L_t^a, t\geq 0)$. We can, and we will assume that $(L_t^a, t\ge 0)$ is normalized to satisfy
\begin{align*}
  \E_x\left[\int_0^\infty e^{-qt}dL_t^a\right]=r_q(a-x),\qquad \forall x\in\R, \ \forall q> 0,
\end{align*}
see, e.g., Section V of \cite{Ber}. It is known that $L^a$ is a positive continuous additive functional whose Revuz measure is given by Dirac's mass at $a,$ $\delta_a$, see, e.g., Section 75 of \cite{Sharpe} for background. Moreover, for every positive finite measure $\nu$, the process
\begin{align*}
  \int_\R L^a_{t}\nu(da),\qquad t\geq0;
\end{align*}
is a positive continuous additive functional, see, e.g., Theorem V.5 of \cite{Ber}.

\subsection{Renormalized Zero Resolvent}
\label{S2.2}
Under the assumption that the process is recurrent and that it satisfies the condition (\ref{ConditionA}),  we define
\begin{align*}
  h_q(x):=r_q(0)-r_q(-x),\qquad x\in\R.
\end{align*}
By (\ref{HP-Lap}), $h_q$ is non-negative. By (\ref{resol-den}), we have
\begin{align}
\label{hq-repre}
  h_q(x)=\frac{1}{\pi}\int_0^\infty \Re \left(\frac{1-e^{i\lambda x}}{q+\Psi(\lambda)}\right)d\lambda,\qquad x\in\R.
\end{align}
For any $x\in \R$, the limit
\begin{align*}
 h(x):=\lim_{q\to 0+}h_q(x),\qquad x\in\R,
\end{align*}
exists and is finite. Moreover, this convergence is uniform on compacts, see Theorem 1.1 of Takeda-Yano \cite{TY}. Actually, the existence of this limit $h$ is classical for symmetric \Levy\ processes, see for instance Salminen-Yor \cite{SY}. For \Levy\ processes not necessarily symmetric, however, Takeda-Yano \cite{TY} proved its existence under the mild condition \eqref{ConditionA}. We call the limiting function $h$ the \emph{renormalized zero resolvent}. It is known that $h$ is non-negative, continuous, and subadditive, see Theorem 1.1 of Takeda-Yano \cite{TY}. 

The renormalized zero resolvent has been explicitly determined for several specific \Levy\ processes. For instance, if $X$ is Brownian motion we have that
 \begin{align*}
   h(x)=|x|,\quad x\in\R;
 \end{align*}
see, e.g., Example 5.1 of \cite{TY}. In the case where $X$ is an $\alpha$-stable process with index $\alpha\in (1,2),$ hence whose \Levy\ measure is of the form
\begin{align*}
  \Pi(dx)=\begin{cases}
    c_+|x|^{-\alpha-1}dx&\text{on}\ (0,\infty),\\
    c_-|x|^{-\alpha-1}dx&\text{on}\ (-\infty,0),
  \end{cases}
\end{align*}
where $c_+,c_-\ge 0$ and $c_++c_->0$; we have
\begin{align*}
  h(x)=\frac{1}{K(\alpha)}(1-\beta \mathrm{sgn}(x))|x|^{\alpha-1},\qquad x\in\R,
\end{align*}
with
\begin{align*}
  \beta:=\frac{c_+-c_-}{c_++c_-},\qquad K(\alpha):=-\frac{(c_++c_-)\pi}{\alpha \tan (\frac{\pi\alpha}{2})}\left(1+\beta^2 \tan^2 \left(\frac{\pi\alpha}{2}\right)\right);
\end{align*}
see Section 5 of Yano \cite{Yano2}. Besides, when $X$ is a spectrally negative \Levy\ process, we have
\begin{align}
\label{h-SNLP}
  h(x)=W(x)-\frac{x}{\E_0[X_1^2]},\qquad x\in\R;
\end{align}
where $W$ denotes the zero-scale function of $X$, see Example 5.2 of Pant\'{i} \cite{Pa}.

In any case, the renormalized zero resolvent is an invariant function for the process killed at $\{0\}$, that is, 
\begin{align*}
 h(x)=\E_x[h(X_t),\ t<T_{\{0\}}],\qquad x\in \R, \ t\geq 0;
\end{align*}
see Theorem 2.2 of Pant\'{i} \cite{Pa}, see also Theorem 1.1 of Yano \cite{Yano3}.

%Salminen-Yor \cite{SY} showed that, for symmetric \Levy\ processes, the renormalized zero resolvent exists under the assumptions that every point is regular for itself and that the process is not a compound Poisson process. In the asymmetric case, however, additional assumptions are required. Yano \cite{Yano2}, Pant\'{i} \cite{Pa}, Tsukada \cite{Tukada}, and Takeda–Yano \cite{TY} established the existence of the renormalized zero resolvent under certain integrability conditions, formulated in terms of the characteristic exponent $\Psi$. In particular, the condition imposed by Takeda–Yano \cite{TY} is the weakest among these assumptions and is stated as follows:
%\begin{align*}
%\label{ConditionA}
%  \int_0^\infty \left|\frac{1}{q+\Psi(\lambda)}\right|d\lambda<\infty\qquad \text{for}\ q>0. \tag{\bf{A}}
%\end{align*}
%For the relative strength of the assumptions in these papers, see Section 1.2 of \cite{TY}.

%From a potential-theoretic point of view, the function $h$ is invariant for the process killed upon hitting the singleton $\{0\}$:
%\begin{align}
%  h(x)=\E_x[h(X_t),\ t<T_{\{0\}}].
%\end{align}
%Yano \cite{Yano3} proved this for symmetric \Levy\ processes, while Pant\'{i} \cite{Pa} proved it for asymmetric \Levy\ processes.  

Moreover, Pant\'{i} \cite{Pa} and Takeda-Yano \cite{TY} proved that the h-transform of the process killed at zero, defined using this function $h$, arises when one considers the \Levy\ process conditioned to avoid zero: for all $x\in\R,$ and $t\geq 0,$ we have
\begin{align*}
\lim_{q\to 0}\P_x(\Lambda |\ \bm{e}_q<T_{\{0\}})=\frac{1}{h(x)}\E_x\left[1_\Lambda\cdot h(X_t),\ t<T_{\{0\}}\right], \qquad \text{for}\ \Lambda\in \F_t;
\end{align*}
where $\bm{e}_q$ is a random variable with an exponential distribution of parameter $q>0$, independent of $X$. Furthermore, we mentioned in the first lines of this document, Salminen-Yor \cite{SY} and Tsukada \cite{Tukada} established Tanaka's formula \eqref{Tanaka-Levy} using the renormalized zero resolvent $h$.

 Iba \cite{Iba} considered extensions of the above results, focusing on invariant functions for processes killed on a compact set $B$. For that end, he defined the function
\begin{align}\label{intvarphi}
  \varphi_B(x):=h(x)-\E_x[h(X_{T_B})],\qquad x\in\R.
\end{align}
In Theorem 1.3 of \cite{Iba}, Iba established that
\begin{align}
\label{phiq}
  \lim_{q\to 0}r_q(0)\P_x(\bm{e}_q<T_B)=\varphi_B(x),\qquad x\in\R,
\end{align}
and $\varphi_B$ is invariant for the process killed at $B$, that is,
\begin{align*}
  \varphi_B(x)=\E_x[\varphi_B(X_t),\ t<T_B],\qquad x\in\R.
\end{align*}
In the case of recurrent stable processes, these facts were also obtained by Port \cite{PStable}. Moreover, in Theorem 1.3 of \cite{Iba} it has also been shown that on top of \eqref{phiq}, we have that the following limit exists: for all $x\in\R,$ and $t\geq 0$
\begin{align*}
  \lim_{q\to 0}\P_x(\Lambda|\ \bm{e}_q<T_B)=\frac{1}{\varphi_B(x)}\E_x[1_\Lambda\cdot \varphi_B(X_t),\ t<T_B]\qquad \text{for all}\ \Lambda\in \F_t.
\end{align*} 

%We recall that the aim of this paper is to extend the one-point Tanaka's formula to compact sets. Let $B$ be a nonempty compact set. Iba \cite{Iba} studied a generalization of (\ref{condi-zero}) to compact sets $B$. More precisely, under the same assumptions as Takeda-Yano \cite{TY}, 
%where
%\begin{align}
%   \varphi_B(x):=h(x)-\E_x[h(X_{T_B})]
%\end{align}
%is an invariant function associated with the process killed at the compact set $B$. 
Knowing the integral expression \eqref{intvarphi} of $\varphi_{B}$ in terms of  $h$ and that $h$ admits the decomposition \eqref{Tanaka-Levy}, it is natural to ask whether $\varphi_B(X_{\cdot})$ also admits a Tanaka-type formula. To answer this positively and provide a precise statement, we still require to recall the notion of $q$-capacity measure. 

From Theorem 42.5 and Proposition 42.13 of \cite{Sato}, we know that for $q>0$, there exists a unique measure $m_B^q$ such that
\begin{align*}
  \E_x[e^{-qT_B}]=\int r_q(y-x)m_B^q(dy).
\end{align*}
Moreover, $m_B^q$ is supported on $B$. As usual, we refer to $m_B^q$ as the $q$-capacity measure of $B$. 

For transient processes, it is known that the $q$-capacity measure converges weakly, as $q\to 0$, to the $0$-capacity measure, which is called the equilibrium measure. For recurrent processes, a normalization is required for the $q$-capacity measure to converge towards a non-degenerate measure. That is the content of the following Proposition, that will turn rather relevant for our developments.  
\begin{prop}
\label{main-prop}
Assume that $X$ is recurrent and (\ref{ConditionA}) holds. We have that there is a unique probability measure $\mu_B$ with support in $B$ that arises as the weak limit
  \begin{align*}
    r_q(0)m_B^q\Rightarrow \mu_B,\qquad q\to 0;
  \end{align*}
and a real-valued constant $k(B),$ such that
\begin{align*}
  \lim_{q\to 0}r_q(0)\Big(r_q(0)m_B^q(B)-1\Big)=k(B).
\end{align*} We will refer to $\mu_B$ as the  \emph{equilibrium measure} and to $k(B)$ as the \emph{Robin constant}.
\end{prop}

%One of the motivations in this paper is to shed further light into 
%
%  
%
%
%
%
%Ignited by the contributions of , in this paper we aim to provide what, for the best of our knowledge, is the first instance of a Tanaka like formula for bounded sets. That is, rather than considering the local time at a point we consider the local time, as Markov process, at a compact set and we provide an explicit formula 
\subsection{Main results}
We have now all the elements to state our main results. The first one, sheds further light on the structure of the function $\varphi_B$ by providing an integral representation of $\varphi_B$ in terms of the renormalized zero resolvent $h$.
\begin{thm}
\label{main1}
Assume that $X$ is recurrent and (\ref{ConditionA}) holds, and let $B$ be a nonempty compact set. The equilibrium measure $\mu_{B}$ and the Robin constant from Proposition~\ref{main-prop} are related to $\varphi_{B}$ by the equation
 \begin{align}
  \label{main1-eq}
     \varphi_B(x)=\int_{y\in B} h(x-y)\mu_B(dy)-k(B),\qquad x\in\R.
  \end{align}
\end{thm}
Port \cite{PStable} obtained a similar representation formula for recurrent stable processes, naming $\mu_B$ the equilibrium measure, and $k(B)$ the Robin constant; we preserve the names in the general setting for consitency. The proof of this theorem is given in Section \ref{S3}. 

The claimed Tanaka-type formula for compact sets will actually be obtained as a Corollary of the following result which happens to be more general than needed and of interest in itself. Indeed, instead of focusing on the local time at the compact set $B,$ we can consider a rather general continuous additive functional $A$ of $X,$ characterized via its Revuz measure $\nu_{A}$. Given such an additive functional, we wonder whether it admits a representation of the type \eqref{Tanaka-Levy}. Namely, whether there exists a function $H^{A}$ and a local martingale $M^{A}$ such that the following relation holds
\begin{align*}
  H^{A}(X_{\cdot})=H^A(X_0)+M^{A}_{\cdot}+A_{\cdot}
\end{align*}
This falls within the deep theory of stochastic calculus for Markov processes and the interested reader can look to the book by Fukushima, Oshima and Takeda~\cite{fukushimaetal} for symmetric Markov processes and to the article by Eisenbaum and Walsh~\cite{EisenbaumWalsh} for real-valued \Levy\ processes, and Walsh \cite{Walsh1} for more general Markov processes, to cite but three influential works in this topic. It is worth emphasizing that the relation to these works is somehow indirect, as in here, rather than starting with a function, say $u,$ and knowing whether some version of It\^{o}'s formula holds for $u(X_{\cdot}),$ which would lead to the appearance of a local martingale and an additive functional, we start with an additive functional $A,$ and build the corresponding function $H^{A}$ and the local martingale. A precise statement is given below.   
\begin{thm}
\label{main2}
Let $(A_t, t\ge 0)$ be a positive continuous additive functional of $X$, and let $\nu$ be its Revuz measure. Assume that $\nu$ is a finite measure and has compact support. Then there exists a martingale $M^\nu$ such that
\begin{align}
\label{main2-eq1}
  \int_\R h(X_t-y)\nu(dy)=\int_\R h(x-y)\nu(dy)+M_t^\nu +A_t,\qquad t\geq 0.
\end{align}
In particular, it holds that
\begin{align}
\label{main2-eq2}
\varphi_B(X_t)=\varphi_B(x)+M_t^{\mu_B}+\int_B L_t^y \mu_B(dy),\qquad t\geq 0.
\end{align}
\end{thm}

The proof of this theorem is given in Section \ref{S4}.

Below we will describe some examples in detail. We will see that when $B$ is a finite set, the equilibrium measure and the Robin constant can be described in terms of a matrix whose entries are given by the renormalized zero resolvent; see Subsection \ref{SS4.1}. For recurrent stable processes and intervals $B$, we compute the equilibrium measure, the Robin constant, and an explicit representation of the martingale part in the Tanaka-type formula; see Subsections \ref{SS5.1} and \ref{SS5.2}. For recurrent
spectrally negative \Levy\ processes and intervals $B$, we compute the equilibrium measure and the Robin constant; see Section \ref{S6}. 

All the above discussion concerns recurrent \Levy\ processes. The transient case will be presented in Section \ref{App}.

\subsection*{Organization of the paper}
The rest of this paper is organized as follows. Our main results, Theorems \ref{main1} and \ref{main2} are proved in Section \ref{S3} and in Section \ref{S4}, respectively. Then in Subsection \ref{SS4.1}, we compute the equilibrium measure and the Robin constant for a finite set. In Section \ref{S5}, we consider recurrent stable processes and the interval case. In particular, in Subsection \ref{SS5.1}, we compute the equilibrium measure and the Robin constant for an interval in the case of a stable process. In Subsection \ref{SS5.2}, we give a representation of the martingale part of Tanaka's formula. In Section \ref{S6}, we consider recurrent spectrally negative \Levy\ processes and describe the equilibrium measure and the Robin constant in the interval case. Finally, in Section \ref{App}, we discuss the transient case.

%%%%%%%%%%%%%%%%%%%%%%%%%%%%%%%%%%%%%%%%%%%%%%%%%%%%%%%%%%%%%%%%%%%%%%%%%%%%%%%%%%%%%%%%%%%%%%%%%%%%%%%%%%%%%%%%%%%%%%%%%%%%%%%%%%%%%%%%%%%%%%%%%%%%%%%%%%%%%%%%%%%%%%%%%%%%%%%%%%%%%%%%%%%%%%%%%%%%%%%%%%%%%%%%%%%%%%%%%%%%%%%%%%%%%%%%%%%%%%%%%%%%%%%%%%%%%%%%%%%%%%%%%%%%%%%%%%%%%%%%%%%%%%%%%%%%%%%%

\section{Equilibrium Measure and Robin Constant for Compact Sets}
\label{S3}
In this section, we prove  Proposition~\ref{main-prop} and Theorem \ref{main1}. We first prove existence, in doing so we will prove the validity of the equation \eqref{main1-eq}, and then we will prove the uniqueness property. We will henceforth assume that $B$ is a compact set.

\begin{proof}[Proof of the existence part in Proposition~\ref{main-prop} and the representation \eqref{main1-eq}]
Recall that for $q>0$, the $q$-capacity measure of $B$, $m_B^q$, exists and satisfies
\begin{align}
\label{q-cap}
  \E_x[e^{-qT_B}]=\int r_q(y-x)m_B^q(dy),\qquad x\in \R.
\end{align}
Moreover, $m_B^q$ is supported on $B$. See, e.g., Theorem 42.5 and Proposition 42.13 of \cite{Sato}. We define the measure
\begin{align*}
  \mu_B^q(dy):=r_q(0) m_B^q(dy).
\end{align*}
Then, by Fubini's theorem and the definition of $h_q$, we have for any $x\in \R$,
\begin{align*}
\label{q-eq}
  \varphi^{q}_B(x)&:=r_q(0)\P_x(\bm{e}_q<T_B)\\
  &=r_q(0)\Big(1-\E_x[e^{-qT_B}]\Big)\\
  &=r_q(0)-\int_B r_q(y-x)\mu_B^q(dy)\\
  &=\int_B h_q(x-y)\mu_B^q(dy)-r_q(0)\Big(\mu_B^q(B)-1\Big).
   \stepcounter{equation}\tag{\theequation}
\end{align*}
So, denoting
\begin{align*}
  k_{q}(B):=r_q(0)\Big(\mu_B^q(B)-1\Big),
\end{align*}
the above expression reads 
\begin{align*}
  \varphi^{q}_B(x)&:=r_q(0)\P_x(\bm{e}_q<T_B)\\
  &=\int_B h_q(x-y)\mu_B^q(dy)-k_{q}(B),\qquad x\in\R;
\end{align*}
which is the pre-limit version of \eqref{main1-eq}.

By non-negativity of $h_q$ and (\ref{hq-repre}), we have
\begin{align*}
\label{hq-bdd}
  h_q(x)&\le h_q(x)+h_q(-x)\\
  &=\frac{2}{\pi}\int_0^\infty \Re \left(\frac{1-\cos (x\lambda)}{q+\Psi(\lambda)}\right)d\lambda\\
  &\le \frac{2}{\pi}\int_0^\infty\left|\frac{1-\cos (x\lambda)}{q+\Psi(\lambda)}\right|d\lambda\\
  &\le \frac{2}{\pi}\int_0^\infty\left|\frac{1-\cos (x\lambda)}{\Psi(\lambda)}\right|d\lambda\\
  &\le \frac{2}{\pi}\int_0^\infty \frac{(\lambda x)^2\wedge 2}{|\Psi(\lambda)|}d\lambda \\
  &=:M(x).
    \stepcounter{equation}\tag{\theequation}
\end{align*}
By Lemma 15.5 of Tsukada \cite{Tukada}, the function $M$ is finite and continuous. Since $\varphi_B^q$ vanishes on $B$ and $h_q$ is continuous, for $x\in B$, we have
\begin{align*}
0\leq r_q(0)\Big(\mu_B^q(B)-1\Big)&=\int_B h_q(x-y)\mu_B^q(dy)\\
&\le \sup_{z\in B-B}h_q(z)\cdot \mu_B^q(B)\\
&\le \sup_{z\in B-B}M(z)\cdot \mu_B^q(B).
\end{align*}
Thus, since $r_q(0)\to \infty$ as $q\to 0$, we have
\begin{align*}
  \mu_B^q(B)\le \frac{r_q(0)}{\displaystyle r_q(0)-\sup_{z\in B-B}M(z)}\to 1,\qquad \text{as}\ q\to 0.
\end{align*}
Hence, $(\mu_B^q(B))_{0<q<1}$ is uniformly bounded, and because $B$ is a compact set, it is a consequence of Prohorov's theorem that $(\mu_B^q)_{0<q<1}$ is tight, see, e.g., Corollary 13.30 of \cite{Klenke}. Therefore, there exist a subsequence $(q_n, n\geq 0)$ and a measure $\mu_B$ such that
\begin{align*}
  \mu_{B}^{q_n}\Rightarrow \mu_B,\qquad \text{as}\ n\to\infty.
\end{align*}
Also, because
\begin{equation*}
0\leq r_{q_n}(0)\Big(\mu_B^{q_n}(B)-1\Big) \leq  \frac{r_{q_n}(0)\sup_{z\in B-B}M(z)}{\displaystyle r_{q_n}(0)-\sup_{z\in B-B}M(z)}, \qquad \forall n\geq0,
 \end{equation*}
 we can extract a further subsequence of $(q_{n^{\prime}}, n^{\prime}\geq 0)$ for which there exists a constant $k(B)$ such that
\begin{equation*}
r_{q_{n^{\prime}}}(0)\Big(\mu_B^{q_{n^{\prime}}}(B)-1\Big)\to k(B), \qquad \text{as}\ n^{\prime}\to\infty.
\end{equation*}

Note that $\mu_B$ is a probability measure supported on $B$. Since $h_{q_n}$ converges to $h$ uniformly on compact sets, we have
\begin{align*}
  &\left|\int_B h_{q_n}(x-y)\mu_B^{q_n}(dy)-\int_B h(x-y)\mu_B(dy)\right|\\
  &\qquad \le \int_B \Big|h_{q_n}(x-y)-h(x-y)\Big|\mu_B^{q_n}(dy)+\left|\int_B h(x-y)\mu_B(dy)-\int_B h(x-y)\mu_B^{q_n}(dy)\right|\\
  &\qquad \le \max_{z\in x-B}|h_{q_n}(z)-h(z)| \cdot \mu_B^{q_n}(B)+\max_{z\in x-B}h(z)\cdot \Big|\mu_B(B)-\mu_B^{q_n}(B)\Big|\\
  &\qquad \to 0,
\end{align*}
as $n\to \infty.$ Therefore, letting $q_{n^{\prime}}\to 0$ in (\ref{q-eq}) together with \eqref{phiq}, we obtain (\ref{main1-eq}).
\end{proof}

Before proving uniqueness, we prepare the following lemma.

\begin{lem}
\label{lem-uniq}
  Let $\nu$ be a finite signed measure with compact support. We define
  \begin{align*}
  \nu H_q(x)&:=\int h_q(x-y)\nu(dy),\qquad x\in\R,\\
    \nu H(x)&:=\int h(x-y)\nu(dy),\qquad x\in\R.
  \end{align*}
  Then the Fourier transform of $\nu H_q$ converges to that of $\nu H$ as $q\to 0$:
  \begin{align*}
    \F(\nu H_q)(\xi)\to \F(\nu H)(\xi).
  \end{align*}
\end{lem}
\begin{proof}
  Since $h_q$ is bounded and 
\begin{align*}
  \lim_{x\to \pm \infty}\frac{h(x)}{|x|}=\frac{1}{\E_0[X_1^2]}\in [0,\infty)
\end{align*}
see Theorem 1.2 of Takeda-Yano \cite{TY},
$\nu H_q$ and $\nu H$ are tempered distributions. Let $\mathscr{S}'$ denote the space of tempered distributions. Since the Fourier transform is continuous on $\mathscr{S}'$, it suffices to show that $\nu H_q\to \nu H$ in $\mathscr{S}'.$ By (\ref{hq-bdd}) and Lemma 15.5 of Tsukada \cite{Tukada}, there exists a constant $C_1>0$ such that
\begin{align*}
h_q(x)\le C_1(1+|x|)^2.
\end{align*}
For any rapidly decreasing function $\phi$ and any $N\in \N$, there exists a constant $C_2>0$ such that
\begin{align*}
  |\phi(x+y)|\le \frac{C_2}{(1+|x|)^N}\qquad \text{for}\ y\in \mathrm{supp}(\nu).
\end{align*}
Taking $N=4$, there exists a constant $C_3>0$ such that
\begin{align*}
 |h_q(x)-h(x)||\phi(x+y)|\le \frac{C_3}{(1+|x|)^2}\in L^1(dx).
\end{align*}
Thus, by the dominated convergence theorem, we have
\begin{align*}
\Big|\langle \nu H_q-\nu H,\phi\rangle\Big|&= \left|\int_\R \phi(x) dx \int_\R \Big(h_q(x-y)-h(x-y)\Big)\nu(dy)\right|\\
&\le \int_\R |\nu|(dy)\int_\R |h_q(x)-h(x)||\phi(x+y)|dx\\
&\to 0,
\end{align*}
as $q\to 0$. This completes the proof.
\end{proof}

\begin{proof}[Proof of the uniqueness part of Theorem \ref{main1}]
Assume that there exist measures $\mu_B, \mu_B'$ and constants $k(B),k'(B)$ such that
\begin{align*}
  \varphi_B(x)=\int h(x-y)\mu_B(dy)-k(B)=\int h(x-y)\mu_B'(dy)-k'(B).
\end{align*}
Set $\mu:=\mu_B-\mu_B'$ and $k:=k(B)-k'(B)$. Then we have
\begin{align}
\label{muH}
  \mu H(x)=\int h(x-y)\mu (dy)=k.
\end{align}
Since $h_q$ converges to $h$ uniformly on compact sets and $\mu$ has compact support, we have the pointwise convergence
\begin{align*}
  \mu H_q(x)=\int h_q(x-y)\mu(dy)\to \int h(x-y)\mu(dy),\qquad x\in\R,
\end{align*}
as $q\to 0$. We take the Fourier transform. By Proposition I.9 of \cite{Ber}, we have
\begin{align*}
\mathscr{F}(\mu H_q)(\xi)&=\int_\R e^{i\xi x} dx \int_\R h_q(x-y)\mu(dy)\\
&=\int_\R e^{i\xi x} dx \int_\R r_q(0)\mu(dy)-\int_\R e^{i\xi x} dx \int_\R r_q(y-x)\mu(dy)\\
&=r_q(0)\mu(\R)\int_\R e^{i\xi x}dx-\int_\R e^{-i\xi y}\mu(dy)\int_{\R}e^{i\xi x}\widehat{r_q}(x)dx\\
&=2\pi r_q(0) \mu(\R)\delta (\xi)-\frac{1}{q+\widehat{\Psi}(\xi)}\cdot \F(\mu)(-\xi),
\end{align*}
where $\delta$ denotes the Dirac delta function, and $\widehat{\cdot}$ denotes quantities associated with the dual process of $X$, viz $\widehat{X}=-X$. Since $\widehat{\Psi}(0)=0$, we have
\begin{align}
  \label{DDF}
  \hat{\Psi}(\xi)\delta(\xi)=0.
\end{align}
Thus, we have
\begin{align*}
\label{Fourier}
(q+\widehat{\Psi}(\xi))\F(\mu H_q)(\xi)&=(q+\widehat{\Psi}(\xi))2\pi r_q(0)\mu(\R)\delta(\xi)-\F(\mu)(-\xi)\\
&=2\pi qr_q(0)\mu(\R)\delta(\xi)-\F(\mu)(-\xi).
  \stepcounter{equation}\tag{\theequation}
\end{align*}
Observe that by Lemma \ref{lem-uniq} and (\ref{muH}), we have
\begin{align*}
\lim_{q\to 0}\F(\mu H_q)(\xi)&=\F(\mu H)(\xi)=\F(k)(\xi)=2\pi k\delta(\xi).
\end{align*}
Thus, since $qr_q(0)\to 0,$ see, e.g., Lemma 15.5 of \cite{Tukada}, we have that the left-hand side of 
\eqref{Fourier} converges to zero by (\ref{DDF}). On the other hand, the right-hand side of
\eqref{Fourier} converges towards $-\F(\mu)(-\xi)$. Therefore, we obtain $\mu=0$, that is, $\mu_B=\mu_B',$ and hence $k(B)=k'(B)$. The proof of the uniqueness in Theorem \ref{main1} is now complete.
\end{proof}

This completes the proof of  Proposition \ref{main-prop}, because the latter shows that the convergence in the proof of the existence part in Proposition~\ref{main-prop} and the representation \eqref{main1-eq} does not depend on the selected subsequence. Standard arguments hence imply the convergence for any sequence. 
%\begin{proof}[Proof of Proposition \ref{main-prop}]
%By standard arguments, we have shown that the convergence in (\ref{subseqlim}) holds for the whole sequence, that is, is established.
%\end{proof}

%%%%%%%%%%%%%%%%%%%%%%%%%%%%%%%%%%%%%%%%%%%%%%%%%%%%%%%%%%%%%%%%%%%%%%%%%%%%%%%%%%%%%%%%%%%%%%%%%%%%%%%%%%%%%%%%%%%%%%%%%%%%%%%%%%%%%%%%%%%%%%%%%%%%%%%%%%%%%%%%%%%%%%%%%%%%%%%%%%%%%%%%%%%%%%%%%%%%%%%%%%%%%%%%%%%%%%%%%%%%%%%%%%%%%%%%%%%%%%%%%%%%%%%%%%%%%%%%%%%%%%%%%%%%%%%%%%%%%%%%%%%%%%%%%%%%%%%%

\section{Tanaka's Formula for Compact Sets}
\label{S4}
Before proving Theorem \ref{main2}, we establish to auxiliary lemmas. 

Let $(A_t, t\ge 0)$ be an additive functional, and let $\nu$ be its Revuz measure. For $q>0$, we define the $q$-potential of $A$ by
\begin{align*}
\label{Revuz}
u_q^A(x)&:=\E_x\left[\int_0^\infty e^{-qs}dA_s\right]=\int_\R r_q(y-x)\nu(dy);
\stepcounter{equation}\tag{\theequation}
\end{align*}
where the second equality is a consequence of Chapter V, Section 1, Theorem 5 in \cite{Ber}.
Since the proof of the following lemma is almost identical to that of Proposition 15.4 of Tsukada \cite{Tukada}, we give only a brief sketch of the proof.

\begin{lem}
\label{main2-lem1}
There exists a martingale $M^{q,\nu}$ such that
\begin{align}
\label{q-tanaka}
  \int_\R h_q(X_t-y)\nu(dy)=\int_\R h_q(x-y)\nu(dy)+M_t^{q,\nu}+q\int_0^t u_q^A(X_s)ds+A_t.
\end{align}
\end{lem}
\begin{proof}[Sketch of the proof]
By the Markov property, we have
\begin{align*}
  \E_x\left[\int_0^\infty e^{-qs}dA_s\Big|\mathscr{F}_t\right]=\int_0^t e^{-qs}dA_s+e^{-qt}u_q^A(X_t).
\end{align*}
By Fubini's theorem, we have
\begin{align*}
  q\int_0^t e^{qs}ds\int_0^s e^{-qu}dA_u=e^{qt}\E_x\left[\int_0^\infty e^{-qs}dA_s\Big|\F_t\right]-u_q^A(X_t)-A_t.
\end{align*}
Thus, we have
\begin{align*}
u_q^A(X_t)=u_q^A(x)+M_t^{q,\nu}+q\int_0^t u_q^A(X_s)ds-A_t,
\end{align*}
where 
\begin{align*}
  M_t^{q,\nu}:=-q\int_0^t e^{qs}\E_x\left[\int_0^\infty e^{-qu}dA_u\Big|\F_s\right]ds+ e^{qt}\E_x\left[\int_0^\infty e^{-qu}dA_u\Big|\F_t\right]-u_q^A(x).
\end{align*}
We can prove that $(M_t^{q,\nu}, t\geq 0)$ is a martingale with zero mean. Finally, by (\ref{Revuz}) and the definition of $h_q$, we obtain (\ref{q-tanaka}).
\end{proof}

To pass to the limit as $q\to 0$ in (\ref{q-tanaka}), we consider the following lemma.

\begin{lem}
\label{main2-lem2}
 Assume that $\nu$ is a finite measure with compact support. Then the following convergence holds:
  \begin{align*}
  \int_\R h_q(X_t-y)\nu(dy)\to \int_\R h(X_t-y)\nu(dy)\qquad \text{a.s.\ and\ in\ }L^1(\P_x).
\end{align*}
\end{lem}
\begin{proof}
  Since $h_q$ converges to $h$ uniformly on compact sets and $\nu$ is a finite measure, we have
  \begin{align*}
    \left|\int_\R h_q(x-y)\nu(dy)-\int_\R h(x-y)\nu(dy)\right|&\le \nu(\R)\sup_{z\in x-\mathrm{supp}(\nu)}|h_q(z)-h(z)|\to 0,
  \end{align*}
  as $q\to 0$. This proves the almost sure convergence.

Next, we prove the $L^1$-convergence. By (\ref{hq-repre}) and the proof of Theorem 15.2 of Tsukada \cite{Tukada}, there exist positive constants $C_2,C_1,C_0$ such that
  \begin{align*}
    &\E_x\Big[|h_q(X_t-y)-h(X_t-y)|\Big]\\
    &\qquad \le \E_x\left[h_q(X_t-y)+\lim_{p\to 0}h_p(X_t-y)\right]\\
    &\qquad \le \frac{4}{\pi}\E_x\left[\int_0^\infty \frac{1-\cos (\lambda(X_t-y))}{|\Psi(\lambda)|}d\lambda\right]\\
    &\qquad \le \frac{8}{\pi}\E_x\left[\int_0^1 \frac{(t\omega(\lambda))^2+(\lambda(x+y))^2}{|\Psi(\lambda)|}d\lambda+\int_1^\infty \left|\frac{2}{\Psi(\lambda)}\right|d\lambda+t\right]\\
    &\qquad = C_2y^2 + C_1 |y|+C_0,
  \end{align*}
  where the right-hand side belongs to $L^1(\nu)$. Thus, by the dominated convergence theorem, we have
  \begin{align*}
 & \E_x\left[\left|\int_\R h_q(X_t-y)\nu(dy)-\int_\R h(X_t-y)\nu(dy)\right|\right]\\
 &\qquad \le \int_\R \E_x\Big[|h_q(X_t-y)-h(X_t-y)|\Big]\nu(dy)\\
 &\qquad \to 0
  \end{align*}
  as $q\to 0$. This completes the proof.
\end{proof}

We now prove Theorem \ref{main2}.

\begin{proof}[Proof of Theorem \ref{main2}]
Since $qr_q(0)\to 0$ and $r_q(z)\le r_q(0)$ for all $z\in \R$, we have
\begin{align*}
  q\int_0^t u_q^A(X_s)ds&=\int_0^t ds\int_\R qr_q(y-X_s)\nu(dy)\le t\nu(\R)\cdot qr_q(0)\to 0.
\end{align*}
Thus, by Lemmas \ref{main2-lem1} and \ref{main2-lem2}, $M_t^{q,\nu}$ also converges almost surely and in $L^1(\P_x)$ as $q\to 0$. We denote the limiting process by $(M_t^\nu, t\ge 0)$. Then $(M_t^\nu, t\ge 0)$ is a martingale (see, e.g., Proposition 1.3 of \cite{CW}). Therefore, letting $q\to 0$ in (\ref{q-tanaka}), we obtain (\ref{main2-eq1}).

Next, we prove (\ref{main2-eq2}). Let $B$ be a compact set and let $\mu_B$ be as in Theorem \ref{main1}. By Theorem V.5 of \cite{Ber}, 
\begin{align*}
 t\mapsto \int_BL_t^y\mu_B(dy)
\end{align*}
is a positive continuous additive functional whose Revuz measure is $\mu_B$. Therefore, applying (\ref{main2-eq1}) with $\nu=\mu_B$ and subtracting $k(B)$ from both sides, we obtain (\ref{main2-eq2}). This completes the proof.
\end{proof}

\section{Examples}
\label{S-Eg}
In this section, we deal with some explicit examples.

\subsection{The Case of a Finite Set}
\label{SS4.1}
In this subsection, we derive explicit expressions for $\mu_B$ and $k(B)$ when $B$ is a finite set.

First, we consider the one-point case.

\begin{prop}
If $B=\{a\}$, then 
\begin{align}
\label{one-pt}
  \mu_{\{a\}}(dy)=\delta_a(dy),\qquad k(\{a\})=0.
\end{align}
\end{prop}
\begin{proof}
  By (\ref{HP-Lap}) and (\ref{q-cap}), we have
  \begin{align*}
    \int_\R r_q(y-x)m_{\{a\}}^q(dy)=\P_x[e^{-qT_{\{a\}}}]=\frac{r_q(a-x)}{r_q(0)}.
  \end{align*}
  Thus, we have
  \begin{align*}
    \mu_{\{a\}}^q(dy)=r_q(0)m_{\{a\}}^q(dy)=\delta_a(dy).
  \end{align*}
This proves (\ref{one-pt}).
\end{proof}

Next, we consider the two-point case.

\begin{prop}
\label{two-prop}
  If $B=\{a,b\}$ with $a\neq b$, then 
  \begin{align}
  \label{two-pt}
     \mu_{\{a,b\}}(dy)&=\frac{h(a-b)}{h(a-b)+h(b-a)}\delta_a(dy)+\frac{h(b-a)}{h(a-b)+h(b-a)}\delta_b(dy),\\
     k(\{a,b\})&=\frac{h(a-b)h(b-a)}{h(a-b)+h(b-a)}.
  \end{align}
\end{prop}
\begin{proof}
   By (\ref{q-cap}) and Lemma 3.5 of Takeda-Yano \cite{TY}, we have
  \begin{align*}
  \label{2pt-equi}
    &\int_\R r_q(y-x)m_{\{a,b\}}^q(dy)\\
    &\qquad =\P_x\Big[e^{-qT_{\{a,b\}}}\Big]\\
    &\qquad =\P_x\Big[e^{-qT_{\{a\}}},\ T_{\{a\}}<T_{\{b\}}\Big]+\P_x\Big[e^{-qT_{\{b\}}},\ T_{\{b\}}<T_{\{a\}}\Big]\\
    &\qquad =\frac{h_q(a-b)+h_q(b-a)-\frac{1}{r_q(0)}\{h_q(x-a)h_q(a-b)+h_q(x-b)h_q(b-a)\}}{h_q(a-b)+h_q(b-a)-\frac{1}{r_q(0)}\{h_q(a-b)h_q(b-a)\}}.
    \stepcounter{equation}\tag{\theequation}
  \end{align*}
A straightforward calculation yields
    \begin{align*}
      m_{\{a,b\}}^q(dy)&=\frac{h_q(a-b)\delta_a(dy)+h_q(b-a)\delta_b(dy)}{r_q(0)^2-r_q(a-b)r_q(b-a)}.
    \end{align*}
Since
\begin{align*}
    \frac{r_q(0)^2-r_q(a-b)r_q(b-a)}{r_q(0)}&=h_q(a-b)+h_q(b-a)-\frac{1}{r_q(0)}\{h_q(a-b)h_q(b-a)\}\\
    &\to h(a-b)+h(b-a),
\end{align*}
we obtain (\ref{two-pt}), multiplying (\ref{2pt-equi}) by $r_q(0)$.
\end{proof}

Finally, we consider an $n$-point set $\{a_1,...,a_n\}$ with $n\ge 2$. 

\begin{prop}
\label{finit-proposition}
Let $n\ge 2$ and $a_1<\cdots <a_n$. Then the following statements hold:
\begin{enumerate}
  \item If $h(x)=0$ for $x<0$, then
  \begin{align*}
    \mu_{\{a_1,...,a_n\}}(dy)=\delta_{a_n}(dy),\qquad k(\{a_1,...,a_n\})=0.
  \end{align*}
   \item If $h(x)=0$ for $x>0$, then
  \begin{align*}
    \mu_{\{a_1,...,a_n\}}(dy)=\delta_{a_1}(dy),\qquad k(\{a_1,...,a_n\})=0.
  \end{align*}
  \item If $h(x)>0$ for $x\neq 0$, then
  \begin{gather*}
    \mu_{\{a_1,...,a_n\}}(dy)=\frac{1}{\bm{1}^T\bm{H}^{-1}\bm{1}}\sum_{i=1}^n (\bm{H}^{-1}\bm{1})_i\delta_{a_i}(dy),\\
    k(\{a_1,...,a_n\})=\frac{1}{\bm{1}^T\bm{H}^{-1}\bm{1}}.
  \end{gather*}
\end{enumerate}
\end{prop}
\begin{Rem}
  By continuity and subadditivity of $h$, there are only the three possibilities in Proposition \ref{finit-proposition}, see, e.g. page 3189 of \cite{Takeda}. For example, spectrally negative \Levy\ processes with infinite second moment belong to case 1, whereas spectrally positive \Levy\ processes with infinite second moment belong to case 2. These follow from the explicit formula (\ref{h-SNLP}) for the renormalized zero resolvent $h$.
\end{Rem}

The proof of the latter Proposition will be derived as a consequence of the forthcoming lemmas.

\begin{lem}
\label{finite-lem1}
If $\emptyset\neq B_1\subset B_2$, then $k(B_1)\le k(B_2)$.
\end{lem}
\begin{proof}
 This follows from the fact that the $q$-capacity $m_B^q(B)$ is monotone in $B$, see, e.g., Proposition 42.12 of \cite{Sato}. Hence, $k(B)$ has the same monotonicity property.
\end{proof}

We define a $n\times n$ matrix
\begin{align*}
  \bm{H}:=\Big(h(a_i-a_j)\Big)_{i,j=1}^n=\begin{pmatrix}
    0&h(a_1-a_2)&\cdots &h(a_1-a_n)\\
    h(a_2-a_1)&0&\cdots &h(a_2-a_n)\\
    \vdots &\vdots &&\vdots \\
    h(a_n-a_1)&h(a_n-a_2)&\cdots &0
  \end{pmatrix}
\end{align*}
and define the $n$-dimensional column vectors
\begin{align*}
  \bm{1}:=\Big(1,...,1\Big)^T,\qquad \bm{0}:=\Big(0,...,0\Big)^T.
\end{align*}
We establish several properties of this matrix.

\begin{lem}
\label{finite-lem2}
We have
  \begin{align*}
      \bm{\beta}^T\bm{H}\bm{\beta}<0,\qquad \text{whenever}\ \bm{1}^T\bm{\beta}=\sum_{j=1}^n \beta_j=0\ \text{and}\ \bm{\beta}\neq \bm{0}.
    \end{align*}
\end{lem}
\begin{proof}
For $i<j$, by (\ref{hq-repre}) and the dominated convergence theorem, we have
\begin{align*}
h(a_i-a_j)+h(a_j-a_i)&=\lim_{q\to 0}\Big(h_q(a_i-a_j)+h_q(a_j-a_i)\Big)\\
&=\lim_{q\to 0}\int_0^\infty \Re \left(\frac{1-\cos (\lambda (a_i-a_j))}{q+\Psi(\lambda)}\right)d\lambda\\
&=\int_0^\infty \Re \left(\frac{1-\cos (\lambda (a_i-a_j))}{\Psi(\lambda)}\right)d\lambda.
\end{align*}
Thus, by the assumptions, we obtain
      \begin{align*}
\bm{\beta}^T\bm{H}\bm{\beta}&=\sum_{i,j=1}^n \beta_i \beta_j h(a_i-a_j)\\
&=\sum_{i<j}^n \beta_i \beta_j \Big(h(a_i-a_j)+h(a_j-a_i)\Big)\\
&=\sum_{i<j}^n \beta_i \beta_j \cdot \frac{2}{\pi}\int_0^\infty \Re \left(\frac{1-\cos (\lambda (a_i-a_j))}{\Psi(\lambda)}\right)d\lambda\\
&= -\frac{1}{\pi}\int_0^\infty \left|\sum_{i=1}^n\beta_i e^{{\mathrm i}\lambda a_i}\right|^2\Re \left(\frac{1}{\Psi(\lambda)}\right)d\lambda\\
&<0.
      \end{align*}
This completes the proof.
\end{proof}

\begin{lem}
\label{finite-lem3}
If $h(x)> 0$ for $x\neq 0$, then $\bm{H}$ is invertible.
\end{lem}
\begin{proof}
   Assume that $\bm{H}\bm{\beta}=\bm{0}$. 
    \begin{enumerate}
      \item Case 1: $\bm{1}^T\bm{\beta}=0$. If $\bm{\beta}\neq \bm{0}$, then by Lemma \ref{finite-lem2}, we have
      \begin{align*}
        0>\bm{\beta}^T \bm{H\beta}=\bm{\beta}^T (\bm{H\beta})=0.
      \end{align*}
      This is a contradiction. Thus, in this case, we must have $\bm{\beta}=\bm{0}$.
      \item Case 2: $\bm{1}^T \bm{\beta}\neq 0$. Replacing $\beta$ by $\frac{\bm{\beta}}{\bm{1}^T\bm{\beta}}$, we may assume without loss of generality that $\bm{1}^T\bm{\beta}=1.$  By Theorem \ref{main1}, we can write
    \begin{align}
    \label{finite-equi}
      \mu_{\{a_1,...,a_n\}}(dy)=\sum_{j=1}^n \alpha_j \delta_{a_j}(dy),\qquad \bm{1}^T\bm{\alpha}=1,\qquad \bm{H}\bm{\alpha}=k(\{a_1,...,a_n\})\bm{1},
    \end{align}
    where
    \begin{align*}
\bm{\alpha}:=\Big(\alpha_1,...,\alpha_n\Big)^T.
    \end{align*} 
Here, the final expression in (\ref{finite-equi}) follows by setting $x=a_i$ in (\ref{main1-eq}).
  Let $\bm{\gamma}:=\bm{\alpha}-\bm{\beta}$. Then, we have $\bm{1}^T\bm{\gamma}=0$ and
      \begin{align*}
\bm{\gamma}^T\bm{H}\bm{\gamma}=\bm{\gamma}^T(\bm{H}\bm{\alpha}-\bm{H\beta})=k(\{a_1,...,a_n\})\bm{\gamma}^T\bm{1}=0.
      \end{align*}
      Thus, Lemma \ref{finite-lem2} implies that $\bm{\gamma}=\bm{0}$, and hence $\bm{\alpha}=\bm{\beta}$. Therefore, we have
      \begin{align*}
k(\{a_1,...,a_n\})\bm{1}=\bm{H}\bm{\alpha}=\bm{H}\bm{\beta}=\bm{0}.
      \end{align*}
      This implies that $k(\{a_1,...,a_n\})=0.$
      However, by Lemma \ref{finite-lem1} and Proposition \ref{two-prop}, we have
      \begin{align*}
        k(\{a_1,...,a_n\})\ge k(\{a_1,a_2\})=\frac{h(a_1-a_2)h(a_2-a_1)}{h(a_1-a_2)+h(a_2-a_1)}>0.
      \end{align*}
      Which is a contradiction. v Consequently, we obtain $\bm{\beta}=\bm{0}$.
    \end{enumerate}
   In either case, it follows that $\bm{H}$ is invertible.
\end{proof}

\begin{lem}
\label{finite-lem4}
If $h(x)> 0$ for $x\neq 0$, then $\bm{1}^T \bm{H}^{-1}\bm{1}\neq 0$. 
\end{lem}
\begin{proof}
  We assume $\bm{1}^T \bm{H}^{-1}\bm{1}= 0$. Let $\bm{\beta}:=\bm{H}^{-1}\bm{1}$. Then we have
    \begin{align*}
\bm{1}^T\bm{\beta}=\bm{1}^T\bm{H}^{-1}\bm{1}=0.
    \end{align*}
Clearly, $\bm{\beta}\neq \bm{0}$. Thus, by Lemma \ref{finite-lem2}, we obtain
    \begin{align*}
      0>\bm{\beta}^T \bm{H}\bm{\beta}=\bm{\beta}^T(\bm{H}\bm{H}^{-1}\bm{1})=\bm{\beta}^T \bm{1}=0.
    \end{align*}
   This is a contradiction and completes the proof.
\end{proof}

We now prove Proposition \ref{finit-proposition}.

\begin{proof}[Proof of Proposition \ref{finit-proposition}]
Recall that we can write
\begin{align*}
      \mu_{\{a_1,...,a_n\}}(dy)=\sum_{j=1}^n \alpha_j \delta_{a_j}(dy),\qquad \bm{1}^T\bm{\alpha}=1,\qquad \bm{H}\bm{\alpha}=k(\{a_1,...,a_n\})\bm{1}.
    \end{align*}
   Then we have the following matrix equation:
    \begin{align*}
      \begin{pmatrix}
        \bm{H} &-\bm{1}\\
        \bm{1}^T& 0
      \end{pmatrix}
      \begin{pmatrix}
        \bm{\alpha}\\
        k(\{a_1,...,a_n\})
      \end{pmatrix}
      =
      \begin{pmatrix}
        \bm{0}\\
        1
      \end{pmatrix}.
    \end{align*}
   The first two claims in Proposition \ref{finit-proposition} are clear, so we only prove the remaining case. We assume $h(x)> 0$ for $x\neq 0$. By Lemmas \ref{finite-lem3} and \ref{finite-lem4}, we can apply the Schur complement formula. Thus, we have
    \begin{align*}
      \begin{pmatrix}
        \bm{\alpha}\\
        k(\{a_1,...,a_n\})
      \end{pmatrix}&=\begin{pmatrix}
        \bm{H} &-\bm{1}\\
        \bm{1}^T& 0
      \end{pmatrix}^{-1} \begin{pmatrix}
        \bm{0}\\
        1
      \end{pmatrix}=\begin{pmatrix}
       \displaystyle \frac{\bm{H}^{-1}\bm{1}}{\bm{1}^T\bm{H}^{-1}\bm{1}}\\
       \ \\
        \displaystyle\frac{1}{\bm{1}^T\bm{H}^{-1}\bm{1}}
      \end{pmatrix}.
    \end{align*}
    This proves the third assertion and completes the proof.
\end{proof}

We provide an alternative representation of the quantities obtained in the Case 3 in the Proposition~\ref{finit-proposition}. Its derivation uses the fact that $\mu_{\{a_1,...,a_n\}}$ can be obtained as the limit of the rescaled $q$-capacity measures.
\begin{cor}
\label{finite-corollary}
If $h(x)>0$ for $x\neq 0$, then
\begin{align*}
   \lim_{q\to 0}\left(r_q(0)\sum_{i=1}^n \bm{R}_q^{-1}(i,j)\right)= \left(\frac{\bm{H}^{-1}\bm{1}}{\bm{1}^T \bm{H}^{-1}\bm{1}}\right)_j,
\end{align*}
where 
\begin{align*}
    \bm{R}_q:=\Big(r_q(a_i-a_j)\Big)_{i,j=1}^n=\begin{pmatrix}
    r_q(0)&r_q(a_1-a_2)&\cdots &r_q(a_1-a_n)\\
    r_q(a_2-a_1)&r_q(0)&\cdots &r_q(a_2-a_n)\\
    \vdots &\vdots &&\vdots\\
    r_q(a_n-a_1)&r_q(a_n-a_2)&\cdots &r_q(0)
  \end{pmatrix}.
  \end{align*}
\end{cor}

\begin{Rem}
We point out that the matrix $\bm{R}_q$ has a close connection with the trace process (or the embedded process). For a process $X$, its trace process $Y$ is the continuous-time Markov chain obtained by time-changing $X$ with the right-continuous inverse of the local time of $X$ on a finite set $\{a_1,...,a_n\}$:
\begin{align*}
  Y_u:=X_{A_u^{-1}}\qquad \text{for}\ u\ge 0,
\end{align*}
where 
\begin{align*}
  A_t:=\alpha_1L_t^{a_1}+\cdots +\alpha_n L_t^{a_n},\qquad A_u^{-1}:=\inf \{t>0;\ A_t>u\};
\end{align*}
with $0<\alpha_1,\ldots, \alpha_n<\infty$ determined by the normalization of the local time chosen. 
Getoor \cite{G} showed that the infinitesimal generator, $\bm{Q},$ of $Y$ is given as the limit of $-\bm{R}_q^{-1}$ as $q\to 0$:
\begin{align*}
  -\lim_{q\to 0}\bm{R}_q^{-1}=\bm{Q},
\end{align*}
see Proposition 5.3 of \cite{G}. Furthermore, Iba \cite{Iba-Qmat} established an explicit formula for each entry of the generator $\bm{Q}$ in terms of the renormalized zero resolvent $h$. More generally, the trace process defined through the local time on a general set is referred to as the boundary process. This process plays an essential role in the description of excursions from the set. For more details, we refer the reader to the work by Kaspi \cite{Kaspi}.
\end{Rem}

\begin{proof}[Proof of Corollary \ref{finite-corollary}] 
Define the functions 
\begin{align*}
  p_i^q(x):=\E_x\Big[e^{-qT_{\{a_i\}}},\ T_{\{a_i\}}=T_{\{a_1,...,a_n\}}\Big],\qquad x\in\R.
\end{align*}
 By (\ref{q-cap}), the $q$-capacity measure $m_{\{a_1,...,a_n\}}^q(dy),$ can be written in terms of these functions as follows
  \begin{align*}
  \label{finite-qcap}
\int r_q(y-x)m_{\{a_1,...,a_n\}}^q(dy)=\E_x\Big[e^{-qT_{\{a_1,...,a_n\}}}\Big]=\sum_{i=1}^n p_i^q(x).
     \stepcounter{equation}\tag{\theequation}
  \end{align*}
Next, we compute $p_i^q(x)$ for $i=1,...,n$ using the strong Markov property. Indeed,  we have by (\ref{HP-Lap}), that 
  \begin{align*}
  p_i^q(x)& =\E_x\Big[e^{-qT_{\{a_i\}}}\Big]-\sum_{\substack{1\le j\le n\\ j\neq i}}\E_x\Big[e^{-qT_{\{a_i\}}},\ T_{\{a_j\}}=T_{\{a_1,...,a_n\}}\Big]\\
  &=\E_x\Big[e^{-qT_{\{a_i\}}}\Big]-\sum_{\substack{1\le j\le n\\ j\neq i}}p_j^q(x)\E_{a_j}\Big[e^{-qT_{\{a_i\}}}\Big]\\
  &=\frac{r_q(a_i-x)}{r_q(0)}-\sum_{\substack{1\le j\le n\\ j\neq i}}p_j^q(x)\frac{r_q(a_i-a_j)}{r_q(0)}.
  \end{align*}
  The above can be expressed in matrix form as follows: for all $x\in\R$
  \begin{align*}
  {\tiny
  \begin{pmatrix}
    p_1^q(x)\\
    p_2^q(x)\\
    \vdots\\
    p_n^q(x)
  \end{pmatrix}=\frac{1}{r_q(0)}
  \begin{pmatrix}
    r_q(a_1-x)\\
    r_q(a_2-x)\\
    \vdots \\
    r_q(a_n-x)
  \end{pmatrix}
  -\frac{1}{r_q(0)}\begin{pmatrix}
    0&r_q(a_1-a_2)&\cdots &r_q(a_1-a_n)\\
    r_q(a_2-a_1)&0&\cdots &r_q(a_2-a_n)\\
    \vdots &\vdots &&\vdots\\
    r_q(a_n-a_1)&r_q(a_n-a_2)&\cdots &0
  \end{pmatrix}
  \begin{pmatrix}
    p_1^q(x)\\
    p_2^q(x)\\
    \vdots\\
    p_n^q(x)
  \end{pmatrix}}.
  \end{align*}
By Theorem 2.4 of Getoor \cite{G}, the matrix $\bm{R}_q$ is invertible. Thus, this matrix equation can be solved, and we have
\begin{align*}
  p_i^q(x)=\sum_{j=1}^n\bm{R}_q^{-1}(i,j)r_q(a_j-x)\qquad \text{for}\ x\in\R, \  i=1,...,n.
\end{align*}
Therefore, by (\ref{finite-qcap}), we obtain
  \begin{align*}
\int r_q(y-x)m_{\{a_1,...,a_n\}}^q(dy)&=\sum_{i=1}^n p_i^q(x)\\
&=\sum_{i=1}^n\sum_{j=1}^n\bm{R}_q^{-1}(i,j)r_q(a_j-x)\\
&=\sum_{j=1}^n\sum_{i=1}^n\bm{R}_q^{-1}(i,j)r_q(a_j-x).
  \end{align*}
Since $m_{\{a_1,...,a_n\}}^q(dy)$ is supported on $\{a_1,...,a_n\}$, we obtain
\begin{align*}
  m_{\{a_1,...,a_n\}}^q(dy)=\sum_{j=1}^n\left(\sum_{i=1}^n \bm{R}_q^{-1}(i,j)\right)\delta_{a_j}(dy).
\end{align*}
Therefore, by comparing the atoms of $\mu_{\{a_1,...,a_n\}}$ and $m_{\{a_1,...,a_n\}}^q$, we obtain the assertion.
\end{proof}

%%%%%%%%%%%%%%%%%%%%%%%%%%%%%%%%%%%%%%%%%%%%%%%%%%%%%%%%%%%%%%%%%%%%%%%%%%%%%%%%%%%%%%%%%%%%%%%%%%%%%%%%%%%%%%%%%%%%%%%%%%%%%%%%%%%%%%%%%%%%%%%%%%%%%%%%%%%%%%%%%%%%%%%%%%%%%%%%%%%%%%%%%%%%%%%%%%%%%%%%%%%%%%%%%%%%%%%%%%%%%%%%%%%%%%%%%%%%%%%%%%%%%%%%%%%%%%%%%%%%%%%%%%%%%%%%%%%%%%%%%%%%%%%%%%%%%%%%

\subsection{Recurrent Stable Processes and the Interval Case}
\label{S5}
Let $\alpha\in (1,2)$. The characteristic exponent $\Psi(\lambda)$ of a recurrent $\alpha$-stable process $(X_t, t\geq 0)$ has the form
\begin{align*}
  \Psi(\lambda)=|\lambda|^\alpha\left(1-i\beta \mathrm{sgn}(\lambda)\tan \frac{\pi\alpha}{2}\right)\qquad \text{for}\ \lambda\in \R,
\end{align*}
where $\beta\in [-1,1]$ (see, e.g., Theorem 14.10 of \cite{Sato}). For a constant $c>0$, a \Levy\ process with characteristic exponent $\Psi^{(c)}(\lambda):=c\Psi(\lambda)$ is also called a stable process. This constant $c$ is called here the normalization constant. For general $c\neq 1$, the corresponding objects will be denoted with a superscript, such as $T_B^{(c)}$. While when $c=1,$ we will omit the superscript. Before stating the main result of this section, we make a brief digression to describe how changing the normalizing constant would change the quantities of interest. 

For any normalization constant $c\neq 1$, we have
\begin{align}
\label{normalization-c}
   (X_t^{(c)}, t\geq 0)\deq (X_{c t}, t\geq 0),\qquad T_{[-1,1]}^{(c)}\deq \frac{1}{c}T_{[-1,1]};
\end{align}
and
\begin{align*}
u_{[-1,1]}(x,y)dy&=\int_0^\infty \P_x\Big(T_{[-1,1]}>t,\ X_t\in dy\Big)dt\\
&=c\int_0^\infty \P_x\Big(T_{[-1,1]}>uc ,\ X_{uc}\in dy\Big)du\\
&=c \int_0^\infty \P_x\Big(T_{[-1,1]}^{(c)}>u,\ X_u^{(c)}\in dy\Big)du\\
&=c\cdot u_{[-1,1]}^{(c)}(x,y)dy.
\end{align*}
We define the positivity parameter by $\rho:=\P_0(X_t\ge 0)$. In this case, $\rho$ ranges over $[1-\frac{1}{\alpha},\frac{1}{\alpha}]$. It is known that $\rho$ can be computed in terms of the parameter $\beta$:
\begin{align}
\label{rho-beta}
  \rho=\frac{1}{2}+\frac{1}{\pi\alpha}\tan^{-1}\left(\beta\tan \frac{\pi\alpha}{2}\right).
\end{align}
By the scaling property, $\rho$ does not depend on the time and hence neither on the choice of normalization. Thus, it follows from (\ref{rho-beta}) that $\beta$ is also independent of the choice of normalization.

Moreover, it is worth to mention that the equilibrium measure $\mu_{[a,b]}(dy)$ does not depend on the choice of the normalization constant, whereas the Robin constant $k([a,b])$ does. Indeed, by (\ref{resol-den}), after multiplying and dividing by $c,$ we obtain the equality, for any $x\in\R$
\begin{align*}
  r_q(x)=\frac{1}{\pi}\int_0^\infty \Re \left(\frac{e^{i\lambda x}}{q+\Psi(\lambda)}\right)d\lambda=\frac{c}{\pi}\int_0^\infty \Re \left(\frac{e^{i\lambda x}}{cq+\Psi^{(c)}(\lambda)}\right)d\lambda=cr_{cq}^{(c)}(x).
\end{align*}
Thus, by (\ref{phiq}) and (\ref{normalization-c}), we have for any $x\in\R,$
\begin{align*}
  \varphi_{[a,b]}(x)=\lim_{q\to 0}r_q(0)\P_x(\bm{e}_q<T_{[a,b]})=\lim_{q\to 0}cr_{cq}^{(c)}(0)\P_x(\bm{e}_q<cT_{[a,b]}^{(c)})=c\varphi_{[a,b]}^{(c)}(x).
\end{align*}
Similarly, we have $h(x)=ch^{(c)}(x),$ for $x\in\R.$ Thus, we obtain that for any $x\in\R,$ we have the equality
\begin{align*}
  \varphi_{[a,b]}(x)&=c\varphi_{[a,b]}^{(c)}(x)\\
  &=\int ch^{(c)}(x-y)\mu_{[a,b]}^{(c)}(dy)-ck^{(c)}([a,b])\\
  &=\int h(x-y)\mu_{[a,b]}^{(c)}(dy)-ck^{(c)}([a,b]).
\end{align*}
Therefore, by the uniqueness of the equilibrium measure and the Robin constant, we obtain
\begin{align*}
  \mu_{[a,b]}(dy)=\mu_{[a,b]}^{(c)}(dy),\qquad k([a,b])=ck^{(c)}([a,b]).
\end{align*}

\subsubsection{The Equilibrium Measure and the Robin Constant}
\label{SS5.1}
In the next result, using results of D\"{o}ring-Kyprianou-Weissmann \cite{DKW} and Port \cite{PStable}, we establish an expression for the equilibrium measure and the Robin constant for the interval $[a,b]$.

\begin{prop}
\label{stable-genralinterval}
Let $a<b$. It holds that
\begin{align*}
  \mu_{[a,b]}(dy)&=\frac{(b-a)^{\alpha-1}}{B(1-\alpha\hat{\rho},1-\alpha\rho)}(b-y)^{-\alpha\rho}(y-a)^{-\alpha\hat{\rho}}1_{(a,b)}dy,\\
  k([a,b])&=-\frac{(b-a)^{\alpha-1}\pi \cos (\pi\alpha(\rho-\frac{1}{2}))}{\Gamma(\alpha)\sin (\pi\alpha)B(1-\alpha\rho,1-\alpha\hat{\rho})}.
\end{align*}
\end{prop}
To prove Proposition \ref{stable-genralinterval}, we first consider the interval $[-1,1]$, and then extend the result to a general interval $[a,b]$ using the scaling property. For that end, we recall results of Port \cite{PStable} and D\"{o}ring-Kyprianou-Weissmann \cite{DKW}. 

The stable processes considered in \cite{PStable} and \cite{DKW} use different normalizations. The normalization constant is $c=1$ in \cite{PStable}, while it is 
\begin{align*}
  c=\cos C_1:=\cos \pi\alpha\left(\rho-\frac{1}{2}\right)
\end{align*}
in \cite{DKW}. In Section 2 of \cite{PStable}, Port shows that the potential measure of the process killed when it hits the interval $[-1,1]$ is absolutely continuous with respect to Lebesgue's measure, with a density that we will denote by $u_{[-1,1]}$ 
\begin{align*}
  u_{[-1,1]}(x,y)dy=\int_0^\infty \P_x(T_{[-1,1]}>t,\ X_t\in dy)dt,\qquad x,y\in\R.
\end{align*}
Moreover, it is proved in Theorem 2 of \cite{PStable} that the following limit holds for any $x\in\R,$
\begin{align}
\label{P-u}
  \lim_{s\to \infty}s^{1-\frac{1}{\alpha}}\P_x(s<T_{[-1,1]})=\frac{1}{C(\alpha,\beta)}\lim_{y\to \infty}u_{[-1,1]}(x,y)=\frac{1}{C(\alpha,\beta)}\varphi_{[-1,1]}(x),
\end{align}
where
\begin{align*}
  C(\alpha,\beta)&=\Gamma\left(1+\frac{1}{\alpha}\right)\cos \left(\frac{1}{\alpha}\tan^{-1}\left(\beta \tan \frac{\pi\alpha}{2}\right)\right)\left(1+\beta^2 \tan^2\frac{\pi\alpha}{2}\right)^{-\frac{1}{2\alpha}}\frac{1}{\sin \frac{\pi}{\alpha}}.
\end{align*} Besides, in Lemma 3.3 of D\"{o}ring-Kyprianou-Weissmann \cite{DKW} it is established that
\begin{align}
\label{DKW-u}
  \lim_{y\to \infty}u_{[-1,1]}^{(\cos C_1)}(x,y)=K(\alpha,\rho)\phi(x),
\end{align}
where\footnote{In D\"{o}ring-Kyprianou-Weissmann \cite{DKW}, the constant $K(\alpha,\rho)$ is defined by the first line of (\ref{const-K}), so the computations from the second line onward are included here for completeness.}
\begin{align*}
\label{const-K}
    K(\alpha,\rho)&:=\frac{2^{1-\alpha}}{\Gamma(\alpha\rho)\Gamma(\alpha\hat{\rho})}\frac{2(1-\alpha\hat{\rho})\Gamma(\alpha\hat{\rho})}{\Gamma(1-\alpha\rho)}\int_1^\infty (v-1)^{\alpha\rho-1}(v+1)^{\alpha\hat{\rho}-2}dv\\
    &=\frac{2^{1-\alpha}}{\Gamma(\alpha\rho)}\frac{2(1-\alpha\hat{\rho})}{\Gamma(1-\alpha\rho)}\int_0^1\left(\frac{2u}{1-u}\right)^{\alpha\rho-1}\left(\frac{2}{1-u}\right)^{\alpha\hat{\rho}-2}\frac{2}{(1-u)^2}du\qquad \left(u=\frac{v-1}{v+1}\right)\\
    &=\frac{1}{\Gamma(\alpha\rho)}\frac{1-\alpha\hat{\rho}}{\Gamma(1-\alpha\rho)}\int_0^1 u^{\alpha\rho -1}(1-u)^{1-\alpha}du\\
     &=\frac{1}{\Gamma(\alpha\rho)}\frac{1-\alpha\hat{\rho}}{\Gamma(1-\alpha\rho)}B(\alpha\rho,2-\alpha)\\
     &=\frac{1}{\Gamma(\alpha\rho)}\frac{1-\alpha\hat{\rho}}{\Gamma(1-\alpha\rho)}\frac{\Gamma(\alpha\rho)\Gamma(2-\alpha)}{\Gamma(2-\alpha\hat{\rho})}\\
     &=\frac{\Gamma(2-\alpha)}{\Gamma(1-\alpha \rho)\Gamma(1-\alpha \hat{\rho})},
        \stepcounter{equation}\tag{\theequation}
\end{align*}
and
\begin{align*}
  \phi(x)&:=\begin{cases}
   \displaystyle \frac{\Gamma(1-\alpha\rho)}{\Gamma(\alpha\hat{\rho})}\int_1^x (z-1)^{\alpha\hat{\rho}-1}(z+1)^{\alpha\rho-1}dz&(x>1),\\
    \displaystyle\frac{\Gamma(1-\alpha\hat{\rho})}{\Gamma(\alpha\rho)}\int_1^{|x|}(z-1)^{\alpha\rho-1}(z+1)^{\alpha\hat{\rho}-1}dz&(x<-1).
  \end{cases}
\end{align*}
Therefore, using (\ref{P-u}) and (\ref{DKW-u}), we obtain 
\begin{align*}
\label{phi-phi}
  \varphi_{[-1,1]}(x)&=\lim_{y\to \infty}u_{[-1,1]}(x,y)\\
  &=\cos C_1 \lim_{y\to \infty}u_{[-1,1]}^{(\cos C_1)}(x,y)\\
  &=\cos C_1\cdot K(\alpha,\rho)\phi(x).
     \stepcounter{equation}\tag{\theequation}
\end{align*}

%To compute $\mu_{[-1,1]}$, we rewrite the parameter $\beta$ which appears in $h(x)$ in terms of $\rho$. 
For the next Lemma, we recall that 
\begin{align*}
  h(x)=\frac{1}{K(\alpha)}(1-\beta \mathrm{sgn}(x))|x|^{\alpha-1},
  \end{align*}
  where
  \begin{align*}
  K(\alpha):=-2\Gamma(\alpha)\cos \frac{\pi\alpha}{2}\left(1+\beta^2\tan^2\left(\frac{\pi\alpha}{2}\right)\right).
\end{align*}

\begin{lem}
\label{interval-lem}
We have
  \begin{align*}
  h(x)=-\frac{\cos C_1}{\Gamma(\alpha)\sin (\pi \alpha)}\times \begin{cases}
    \sin (\pi\alpha\hat{\rho})|x|^{\alpha-1},&x>0,\\
    \sin (\pi \alpha \rho)|x|^{\alpha-1},&x<0.
  \end{cases}
\end{align*}
\end{lem}
\begin{proof}
  By (\ref{rho-beta}), we have
  \begin{align*}
    \beta=\frac{\tan(\pi\alpha(\rho-\frac{1}{2}))}{\tan \frac{\pi\alpha}{2}}=:\frac{\tan C_1}{\tan C_2}.
  \end{align*}
We will only show the result when $x>0$, the proof in the other case follows by symmetry. In this case, $\mathrm{sgn}(x)=1$. Hence, it suffices to observe that
\begin{align*}
  \frac{1-\beta}{K(\alpha)}&=\left(1-\frac{\tan C_1}{\tan C_2}\right)\frac{1}{-2\Gamma(\alpha)\cos C_2(1+\tan^2 C_1)}\\
  &=-\frac{\sin C_2\cos C_1- \cos C_2\sin C_1}{\sin C_2\cos C_1}\cdot \frac{\cos^2 C_1}{2\Gamma(\alpha)\cos C_2}\\
  &=-\frac{\sin (C_2-C_1)}{\sin 2C_2}\cdot \frac{\cos C_1}{\Gamma(\alpha)}\\
  &=-\frac{\cos C_1}{\Gamma(\alpha)\sin (\pi \alpha)}\cdot \sin (\pi\alpha\hat{\rho}).
\end{align*}
\end{proof}

To compute $\mu_{[-1,1]}$ and $k([-1,1])$, we observe that measure in Proposition~\ref{stable-genralinterval} clearly defines a probability measure, that we denote by $\mu$. By the uniqueness of the equilibrium measure and the Robin constant, it is enough to show that
\begin{align}
\label{int-prop-aim}
  \varphi_{[-1,1]}(x)=\int h(x-y)\mu(dy)-\text{a\ constant}.
\end{align}
This will be achieved in the two following lemmas. Since their proofs are based on elementary calculus arguments, they will be postponed to Appendix \ref{App-lem}.

%\begin{lem}
%\label{interval-prop1}
%It holds that
%\begin{align}
%\label{mu-interval}
%  \mu_{[-1,1]}(dy)&=\frac{(1-y)^{-\alpha \rho}(1+y)^{-\alpha \hat{\rho}}}{2^{1-\alpha}B(1-\alpha\hat{\rho},1-\alpha\rho)}\cdot 1_{(-1,1)}(y)dy,\\
%  k([-1,1])&=-\frac{2^{\alpha-1}\pi \cos (\pi\alpha(\rho-\frac{1}{2}))}{\Gamma(\alpha)\sin (\pi\alpha)B(1-\alpha\rho,1-\alpha\hat{\rho})}.
%\end{align}
%\end{lem}

\begin{lem}
\label{Stable-lemma1}
For $x\ge 1$ or $x\le -1$, we have
\begin{align}
\label{Stable-lemma1-eq}
  \int_\R h(x-y)\mu(dy)=\varphi_{[-1,1]}(x)-\frac{\pi \cos C_1}{2^{1-\alpha}B(1-\alpha\hat{\rho},1-\alpha\rho)\Gamma(\alpha)\sin (\pi\alpha)}.
\end{align}
\end{lem}

Note that $x\mapsto \int h(x-y)\mu(dy)$ is continuous on $\R$. To prove (\ref{int-prop-aim}), it remains to show the following:

\begin{lem}
\label{Stable-lemma2}
The function $x\mapsto \int h(x-y)\mu(dy)$ is constant on $[-1,1]$.
\end{lem}

Now, we compute the equilibrium measure and the Robin constant for a general interval $[a,b]$.

\begin{proof}[Proof of Proposition \ref{stable-genralinterval}]
By (\ref{P-u}) and the scaling property, we have that for any $x\in\R,$ the following equality holds
\begin{align*}
\varphi_{[a,b]}(x)&=\lim_{s\to \infty}C(\alpha,\beta)s^{1-\frac{1}{\alpha}}\P_x(s<T_{[a,b]})\\
&=\lim_{s\to \infty}C(\alpha,\beta)s^{1-\frac{1}{\alpha}}\P_{x-\frac{a+b}{2}}\left(s<T_{[-\frac{b-a}{2},\frac{b-a}{2}]}\right)\\
&=\lim_{s\to \infty}C(\alpha,\beta)s^{1-\frac{1}{\alpha}}\P_{\frac{2}{b-a}(x-\frac{a+b}{2})}\left(\left(\frac{b-a}{2}\right)^{-\alpha}s<T_{[-1,1]}\right)\\
&=\lim_{s\to \infty}C(\alpha,\beta)\left(\frac{b-a}{2}\right)^{\alpha-1}s^{1-\frac{1}{\alpha}}\P_{\frac{2}{b-a}(x-\frac{a+b}{2})}\left(s<T_{[-1,1]}\right)\\
&=\left(\frac{b-a}{2}\right)^{\alpha-1}\varphi_{[-1,1]}\left(\frac{2}{b-a}x-\frac{a+b}{b-a}\right).
\end{align*}
Thus, by Lemmas~\ref{Stable-lemma1} and \ref{Stable-lemma2}, we have for any $x\in\R$
\begin{align*}
\varphi_{[a,b]}(x)&=\left(\frac{b-a}{2}\right)^{\alpha-1}\int_{[-1,1]}h\left(\frac{2}{b-a}x-\frac{a+b}{b-a}-y\right)\mu_{[-1,1]}(dy)\\
&\qquad -\left(\frac{b-a}{2}\right)^{\alpha-1}k([-1,1]).
\end{align*}
Making the change of variable $y=\frac{2z-a-b}{b-a}$, we have for any $x\in\R$
\begin{align*}
&\int_{[-1,1]}h\left(\frac{2}{b-a}x-\frac{a+b}{b-a}-y\right)\mu_{[-1,1]}(dy)\\
&\qquad =\int_{[-1,1]}h\left(\frac{2}{b-a}x-\frac{a+b}{b-a}-y\right)\frac{(1-y)^{-\alpha\rho}(1+y)^{-\alpha\hat{\rho}}}{2^{1-\alpha}B(1-\alpha\hat{\rho},1-\alpha\rho)}dy\\
&\qquad =\int_{[a,b]}h\left(\frac{2}{b-a}(x-z)\right)\frac{(\frac{2}{b-a}(b-z))^{-\alpha\rho}(\frac{2}{b-a}(z-a))^{-\alpha\hat{\rho}}}{2^{1-\alpha}B(1-\alpha\hat{\rho},1-\alpha\rho)}\frac{2}{b-a}dz\\
&\qquad =2^{\alpha-1}\int_{[a,b]}h(x-z)\frac{(b-z)^{-\alpha\rho}(z-a)^{-\alpha\hat{\rho}}}{B(1-\alpha\hat{\rho},1-\alpha\rho)}dz.
\end{align*}
Thus, we conclude that for any $x\in\R$
\begin{align*}
  \varphi_{[a,b]}(x)&=\int_{[a,b]}h(x-z)\cdot \frac{(b-a)^{\alpha-1}}{B(1-\alpha\hat{\rho},1-\alpha\rho)}(b-z)^{-\alpha\rho}(z-a)^{-\alpha\hat{\rho}}dz\\
  &\qquad -\left(\frac{b-a}{2}\right)^{\alpha-1}k([-1,1]).
\end{align*}
Therefore, the assertion follows from the uniqueness established in Theorem \ref{main1}.
\end{proof}

%\begin{Rem}
%Letting $a=0$ and $b=1$, we have
%\begin{align}
%\mu_{[0,1]}(dy)&=\frac{1}{B(1-\alpha\hat{\rho},1-\alpha\rho)}(1-y)^{-\alpha\rho}y^{-\alpha\hat{\rho}}1_{(0,1)}dy.
%\end{align}
%Thus, $\mu_{[0,1]}$ has a Beta distribution with parameters $1-\alpha{\rho}$ and $1-\alpha\hat{\rho}$.
%\end{Rem}

\subsubsection{Representation of the Martingale Part}
\label{SS5.2}
In this subsection, we use results of Tsukada \cite{T2} to derive a representation of the martingale part in Tanaka's formula for an interval. In this subsection, the normalization constant plays no essential role. Hence, it will be omitted from the notation hereafter.

We first recall the following result of Tsukada \cite{T2}. 
\begin{prop}[Theorem 3.1 of \cite{T2}]
For all $a\in \R$ and $t\ge 0$, it holds that
\begin{align}
\label{Tsu-Tanaka}
  h(X_t-a)=h(x-a)+M_t^a+L_t^a,
\end{align}
where
\begin{align*}
  M_t^a:=\int_0^t \int_{\R\setminus \{0\}}\Big(h(X_{s-}-a+y)-h(X_{s-}-a)\Big)\tilde{N}(ds,dy)
\end{align*}
is an $L^2$-martingale. Here, $\tilde{N}$ denotes the compensated Poisson random measure.
\end{prop}

By extending this theorem, we give a representation of the martingale in Tanaka's formula for an interval.

\begin{prop}
\label{main-martingale}
Let $B$ be a bounded closed interval. It holds that
\begin{align*}
\varphi_B(X_t)=\varphi_B(x)+\int_0^t\int_{\R\setminus \{0\}}\Big(\varphi_B(X_{s-}+y)-\varphi_B(X_{s-})\Big)\tilde{N}(ds,dy)+\int_B L_t^a \mu_B(da)
\end{align*}
for every $t$.
\end{prop}

Intuitively, one may expect that (\ref{main2-eq2}) can be obtained by integrating (\ref{Tsu-Tanaka}) with respect to $\mu_B$, which requires interchanging the integration order, and hence of a version of the stochastic Fubini's theorem. But because the required integration is in the space variable, this does not follow from classical results. In what follows we provide a justification. 

For that end, we define the space-time process
\begin{align*}
  H(s;a,y):=h(X_{s-}-a+y)-h(X_{s-}-a), \qquad s\geq 0, a, y\in\R
\end{align*}
Since the process $X_{s-}$ is predictable and $h$ is continuous, $H$ is $\mathscr{P}\otimes \B(B)\otimes \B(\R)$-measurable, where $\mathscr{P}$ denotes the predictable $\sigma$-algebra on $\R_+\times \Omega$. We define 
\begin{align*}
  \overline{H}(s;y):=\int_B H(s;a,y)\mu_B(da), \qquad s\geq 0, a\in\R.
\end{align*}
Then the following martingale can be constructed.

\begin{lem}
The stochastic process
\begin{align*}
  M_t^B:=\int_0^t \int_{\R\setminus \{0\}} \overline{H}(s;y)\tilde{N}(ds,dy)
\end{align*}
is well-defined. Moreover, for any fixed $T>0$, $(M_t^B)_{0\le t\le T}$ is an $L^2$-martingale.
\end{lem}
\begin{proof}
It suffices to show $\overline{H}\in L^2(\P_x(d\omega)\otimes ds|_{[0,T]}\otimes \Pi(dy))$, where $\Pi$ denotes the \Levy\ measure. By Jensen's inequality and (3.13) of Tsukada \cite{T2}, there exist positive constants $C(\alpha)$ and $C'(\alpha)$ such that
\begin{align*}
  &\E_x\left[\int_0^T\int_{\R\setminus \{0\}}\overline{H}^2(s;y)\Pi(dy)ds\right]\\
  &\qquad \le \int_B\E_x\left[\int_0^T\int_{\R\setminus \{0\}}{H}^2(s;a,y)\Pi(dy)ds\right]\mu_B(da)\\
  &\qquad \le \int_B \Big(C(\alpha)T^{\frac{2\alpha-\e_0-2}{\alpha}}+ C'(\alpha)T\Big)\mu_B(da)\\
  &\qquad =C(\alpha)T^{\frac{2\alpha-\e_0-2}{\alpha}}+ C'(\alpha)T<\infty,
  \stepcounter{equation}\tag{\theequation} 
\end{align*}
where $0<\e_0<(\alpha-1)\wedge (2-\alpha).$ Therefore, $M^B$ is well-defined and its $L^2$-integrability is established. The fact that it is a martingale follows from standard results for \Levy\ processes, see, e.g., Theorem 4.2.3 of \cite{App}. This completes the proof.
\end{proof}

Since $B$ is an interval, it follows from the previous subsection that $\mu_B$ has a density function. Next, we consider a Riemann–Stieltjes-type approximation. Let $B=[c,d]$ and let
\begin{align}
  \pi^{(n)}:\ c=x_0^{(n)}<x_1^{(n)}<\cdots <x_n^{(n)}=d,
\end{align}
be a finite partition of $B$. We denote $I_i^{(n)}:=[x_{i-1}^{(n)},x_i^{(n)})$. We assume that the mesh $\|\pi^{(n)}\|:=\max |I_i^{(n)}|$ converges to zero. For each $i$, we choose a representative point $\xi_i^{(n)}\in I_i^{(n)}$. We define
\begin{align}
S_t^{(n)}:=\sum_{i=1}^nM_t^{\xi_i^{(n)}}\mu_B(I_i^{(n)}).
\end{align}
Clearly, $(S_t^{(n)})_{0\le t\le T}$ is an $L^2$-martingale. To study the convergence of $(S^{(n)})$, we consider the following Lemma.

\begin{lem}
\label{L2-conti}
 The mapping $a\mapsto H(\cdot;a,\cdot)$ is continuous from $B$ into $L^2(\P_x(d\omega)\otimes ds|_{[0,T]}\otimes \Pi(dy))$.
\end{lem}
\begin{proof}
Suppose that $a_n\to a$. By \ito's isometry for $\tilde{N}$ (see, e.g., Theorem 4.2.3 of \cite{App}) and (\ref{Tsu-Tanaka}), we have
  \begin{align*}
  \label{H-cont-eq}
  &\E_x\left[\int_0^T \int_{\R\setminus \{0\}}\Big|H(s;a_n,y)-H(s;a,y)\Big|^2\Pi(dy)ds\right]\\
  &\qquad =\E_x\left[\left|\int_0^T \int_{\R\setminus \{0\}}\Big(H(s;a_n,y)-H(s;a,y)\Big)\tilde{N}(ds,dy)\right|^2\right]\\
  &\qquad =\E_x\Big[|M_T^{a_n}-M_T^a|^2\Big]\\
  &\qquad \le 3\E_x\Big[|h(X_T-a_n)-h(X_T-a)|^2\Big]+3\E_x\Big[|L_T^{a_n}-L_T^{a}|^2\Big]\\
  &\qquad \qquad +3|h(x-a_n)-h(x-a)|^2.
      \stepcounter{equation}\tag{\theequation}
  \end{align*}
Since $h$ is continuous, the third term on the right-hand side of (\ref{H-cont-eq}) converges to $0$. By the H\"{o}lder continuity of $a\mapsto L_t^a$ (see (4.4) of Boylan \cite{Boy}), there exists a positive constant $K$ such that
\begin{align*}
  \E_x\Big[|L_T^{a_n}-L_T^{a}|^2\Big]\le  K \Big|\log |a_n-a|\Big|^2|a_n-a|^{\alpha-1}\to 0,
\end{align*}
as $n\to \infty$. Finally, we consider the first term on the right-hand side of (\ref{H-cont-eq}). By the definition of $h$, there exists a positive constant $C$ such that
 \begin{align*}
    |h(X_T-a_n)-h(X_T-a)|^2&\le C(1+|X_T|^{2(\alpha-1)}).
    \end{align*}
    Since $2(\alpha-1)<\alpha$ and stable processes have a finite $\eta$-th moment for every $\eta<\alpha$, the right-hand side is in $L^1(\P_x)$. Thus, by the dominated convergence theorem and the continuity of $h$, the first term on the right-hand side of (\ref{H-cont-eq}) converges to $0$. This completes the proof.
\end{proof}

Let $T>0,$ arbitrarily fixed and denote by $\mathscr{M}_T^2$ the space of real-valued, zero-mean, almost surely right-continuous, $L^2$-martingales on the finite time interval $[0,T]$. It is known that $\mathscr{M}_T^2$ is a Hilbert space with inner product $\langle M,N\rangle=\E_x[M_TN_T].$ We now study the convergence of $(S^{(n)})$ in this space.

\begin{lem}
\label{M2-convergence}
The sequence of $L^2$-martingales $(S_t^{(n)})_{0\le t\le T}$ converges to $(M_t^B)_{0\le t\le T}$ in $\mathscr{M}_T^2$.
\end{lem}
\begin{proof}
  We define
  \begin{align*}
    H_n(s;y):=\sum_{i=1}^n H(s;\xi_i^{(n)},y)\mu_B(I_i^{(n)}).
  \end{align*}
By \ito's isometry for $\tilde{N}$, we have
  \begin{align*}
    &\langle M^B-S^{(n)},M^B-S^{(n)}\rangle\\
    &\qquad =\E_x\Big[|M_T^B-S_T^{(n)}|^2\Big]\\
    &\qquad =\E_x\left[\left|\int_0^T\int_{\R\setminus \{0\}}\Big(\overline{H}(s;y)-H_n(s;y)\Big)\tilde{N}(ds,dy)\right|^2\right]\\
    &\qquad =\E_x\left[\int_0^T \int_{\R\setminus \{0\}}\Big|\overline{H}(s;y)-H_n(s;y)\Big|^2\Pi(dy)ds\right].
  \end{align*}
  By Jensen's inequality, we have
  \begin{align*}
  |\overline{H}-H_n|^2&=\left|\sum_{i=1}^n\int_{I_i^{(n)}}\Big(H(s;a,y)-H(s;\xi_i^{(n)},y)\Big)\mu_B(da)\right|^2\\
  &\le \sum_{i=1}^n \int_{I_i^{(n)}}\Big|H(s;a,y)-H(s;\xi_i^{(n)},y)\Big|^2\mu_B(da).
  \end{align*}
  We define
  \begin{align*}
    w(\delta):=\sup_{a,b\in B,\ |a-b|\le \delta}\E_x\left[\int_0^T \int_{\R\setminus \{0\}}\Big|H(s;a,y)-H(s;b,y)\Big|^2\Pi(dy)ds\right].
  \end{align*}
  Since $B$ is a compact set, Lemma \ref{L2-conti} implies that $a\mapsto H(\cdot;a,\cdot)$ is uniformly continuous on $B$ as an $L^2(\P_x(d\omega)\otimes ds|_{[0,T]}\otimes \Pi(dy))$-valued map, so we have $w(\delta)\to 0$ as $\delta\to 0$. Thus, we have
  \begin{align*}
     \langle M^B-S^{(n)},M^B-S^{(n)}\rangle&\le \sum_{i=1}^n\int_{I_i^{(n)}}w(|I_i^{(n)}|)\mu_B(da)\le w(\|\pi^{(n)}\|)\to 0
  \end{align*}
  as $n\to \infty$. This completes the proof.
\end{proof}

Consequently, by Doob's inequality and Lemma \ref{M2-convergence}, we have
\begin{align*}
  \E_x\left[\sup_{0\le t\le T}|M_t^B-S_t^{(n)}|^2\right]\le 4\E_x\Big[|M_T^B-S_T^{(n)}|^2\Big]\to 0.
\end{align*}
In particular, there exists a subsequence $(n_k, k\geq 1)$ such that the following uniform almost sure convergence holds.
\begin{align}
\label{Sn-uas}
  \sup_{0\le t\le T}|M_t^B-S_t^{(n_k)}|\to 0\qquad \text{a.s.}
\end{align}

We are now ready to prove Proposition \ref{main-martingale}.

\begin{proof}[The proof of Proposition \ref{main-martingale}]
  Substituting $a=\xi_i^{(n)}$ into (\ref{Tsu-Tanaka}), multiplying both sides by $\mu_B(I_i^{(n)})$, and summing over $i$, we have
  \begin{align}
\label{Riemannsum-Tanaka}
  \sum_{i=1}^n h(X_t-\xi_i^{(n)})\mu_B(I_i^{(n)})=\sum_{i=1}^n h(x-\xi_i^{(n)})\mu_B(I_i^{(n)})+S_t^{(n)}+\sum_{i=1}^nL_t^{\xi_i^{(n)}}\mu_B(I_i^{(n)}).
\end{align}
Since $h$ is continuous and $B$ is a bounded interval, the first term on the right-hand side of (\ref{Riemannsum-Tanaka}) satisfies
  \begin{align*}
    \sum_{i=1}^n h(x-\xi_i^{(n)})\mu_B(I_i^{(n)})\to \int_B h(x-a)\mu_B(da).
  \end{align*}
Next, we consider the left-hand side of (\ref{Riemannsum-Tanaka}). Since $X$ has \cadlag\ paths, there exists a random $R<\infty$ such that
  \begin{align*}
    X_t-a\in [-R,R]\qquad (0\le t\le T,\ a\in B).
  \end{align*}
We define
  \begin{align*}
    w_h(\delta):=\sup_{0\le t\le T,\ a,b\in B,\ |a-b|\le \delta}\Big|h(X_t-a)-h(X_t-b)\Big|.
  \end{align*}
Since $h$ is uniformly continuous on $[-R,R]$, we have almost surely
  \begin{align*}
    &\sup_{0\le t\le T}\left|\sum_{i=1}^n h(X_t-\xi_i^{(n)})\mu_B(I_i^{(n)})-\int_B h(X_t-a)\mu_B(da)\right|\\
    &\qquad \le \sup_{0\le t\le T}\sum_{i=1}^n \int_{I_i^{(n)}}\Big|h(X_t-\xi_i^{(n)})-h(X_t-a)\Big|\mu_B(da)\\
    &\qquad \le \sum_{i=1}^n \int_{I_i^{(n)}}w_h(|I_i^{(n)}|)\mu_B(da)\\
    &\qquad \le w_h(\|\pi^{(n)}\|)\to 0.
  \end{align*}

Finally, we consider the third term on the right-hand side of (\ref{Riemannsum-Tanaka}). For, we define
  \begin{align*}
    w_L(\delta):=\sup_{0\le t\le T,\ a,b\in B,\ |a-b|\le \delta}|L_t^a-L_t^b|.
  \end{align*}
  Note that a recurrent stable process has jointly continuous local times almost surely (see, e.g. Theorem V.15 of \cite{Ber}).
  Since $(t,a)\mapsto L_t^a$ is uniformly continuous on the compact set $[0,T]\times B$, we have
  \begin{align*}
   &\sup_{0\le t\le T}\left|\sum_{i=1}^n L_t^{\xi_i^{(n)}}\mu_B(I_i^{(n)})-\int_B L_t^a\mu_B(da)\right|\\
   &\qquad \le\sup_{0\le t\le T}\sum_{i=1}^n \int_{I_i^{(n)}}\Big|L_t^{\xi_i^{(n)}}-L_t^a\Big|\mu_B(da)\\
   &\qquad \le \sum_{i=1}^n \int_{I_i^{(n)}}w_L(|I_i^{(n)}|)\mu_B(da)\\
   &\qquad \le w_L(\|\pi^{(n)}\|)\to 0.
  \end{align*}
Hence, passing to the limit in (\ref{Riemannsum-Tanaka}) along the subsequence $(n_k, k\geq 1)$ where \eqref{Sn-uas} holds, then we obtain that for all $0\leq t \leq T,$
\begin{align*}
  \int_B h(X_t-a)\mu_B(da)=\int_B h(x-a)\mu_B(da)+M_t^B+\int_B L_t^a \mu_B(da).
\end{align*}
Subtracting $k(B)$ from both sides yields the claim. 
\end{proof}

%%%%%%%%%%%%%%%%%%%%%%%%%%%%%%%%%%%%%%%%%%%%%%%%%%%%%%%%%%%%%%%%%%%%%%%%%%%%%%%%%%%%%%%%%%%%%%%%%%%%%%%%%%%%%%%%%%%%%%%%%%%%%%%%%%%%%%%%%%%%%%%%%%%%%%%%%%%%%%%%%%%%%%%%%%%%%%%%%%%%%%%%%%%%%%%%%%%%%%%%%%%%%%%%%%%%%%%%%%%%%%%%%%%%%%%%%%%%%%%%%%%%%%%%%%%%%%%%%%%%%%%%%%%%%%%%%%%%%%%%%%%%%%%%%%%%%%%

\subsection{Recurrent Spectrally Negative \Levy\ Processes and the Interval Case}
\label{S6}
Let $X$ be a recurrent spectrally negative \Levy\ process. Its \Levy\ measure is supported by $(-\infty,0)$. Then recall that in this case
\begin{align*}
  h(x)=W(x)-\frac{x}{\E_0[X_1^2]},\qquad x\in \R,
\end{align*}
where $W$ denotes the zero-scale function of $X$. Note that $W(0)=0$, since $X$ has paths of unbounded variation, see, e.g., Lemma 8.6 of \cite{Kyp}.

In this case, the equilibrium measure and the Robin constant for the interval $B=[a,b]$ are given as follows:

\begin{prop}
If $\E_0[X_1^2]=\infty$, then we have
\begin{align*}
  \mu_{[a,b]}(dy)&:=\delta_b(dy),\\
  k([a,b])&:=0.
\end{align*}
If $\E_0[X_1^2]<\infty$, then we have on $[a,b]$ 
\begin{align*}
\label{SN-equibri}
  \mu_{[a,b]}(dy)&=\frac{\sigma^2}{2\E_0[X_1^2]}\delta_a(dy)+\frac{1}{\E_0[X_1^2]}\overline{\overline{\Pi}}(a-y)1_{(a,b)}dy\\
  &\qquad +\left(\frac{1}{2}+\frac{1}{2\E_0[X_1^2]}\int_{(-\infty,a-b)}(a-b-x)^2 \Pi(dx)\right)\delta_b(dy),
   \stepcounter{equation}\tag{\theequation}
\end{align*}and 
\begin{align*}
\label{SN-Robinconst}
k([a,b])&=\frac{b-a}{2\E_0[X_1^2]}+\frac{1}{(\E_0[X_1^2])^2}\int_{(a,b)}(y-a)\overline{\overline{\Pi}}(a-y)dy\\
&\qquad +\frac{b-a}{2(\E_0[X_1^2])^2}\int_{(-\infty,a-b)}(a-b-x)^2 \Pi(dx),
  \stepcounter{equation}\tag{\theequation}
\end{align*}
where $\Pi$ denotes the \Levy\ measure, $\overline{\Pi}(y):=\Pi(-\infty,y)$, and $\overline{\overline{\Pi}}(z):=\int_{-\infty}^z\overline{\Pi}(y)dy$.
\end{prop}
\begin{proof}
First, we assume that $\E_0[X_1^2]=\infty$. In this case, we have $h(x)=W(x)$. For $x\in [a,b]$, we have
\begin{align*}
  k([a,b])=\int_{[a,b]} h(x-y)\mu_{[a,b]}(dy)=\int_{[a,x)} W(x-y)\mu_{[a,b]}(dy).
\end{align*}
Putting $x=a$, the right-hand side of the above is equal to zero. Thus, we have $k([a,b])=0$ and $\mu([a,b))=0$. Since $\mu_{[a,b]}$ is a probability measure, it follows that $\mu_{[a,b]}=\delta_b.$

Next, we assume that $\E_0[X_1^2]<\infty.$ For $x\in [a,b]$, we have
\begin{align*}
  k([a,b])&=\int_{[a,b]} h(x-y)\mu_{[a,b]}(dy)\\
  &=\int_{[a,x)} W(x-y)\mu_{[a,b]}(dy)-\frac{1}{\E_0[X_1^2]}\int_{[a,b]} (x-y)\mu_{[a,b]}(dy).
\end{align*}
Putting $x=a$, we have
\begin{align}
\label{SN-Robin}
  k([a,b])=\frac{1}{\E_0[X_1^2]}\int_{[a,b]} (y-a)\mu_{[a,b]}(dy).
\end{align}
Thus, we have
\begin{align}
\label{SN-equibrium}
  \int_{[a,x)} W(x-y)\mu_{[a,b]}(dy)=\frac{x-a}{\E_0[X_1^2]}\qquad \text{for}\ x\in [a,b].
\end{align}
To identify the measure $\mu_{[a,b]}$, we consider the following extension of (\ref{SN-equibrium}) to $x\ge a$:
\begin{align}
\label{SN-formally}
  \int_{[a,x)} W(x-y)\mu(dy)=\frac{x-a}{\E_0[X_1^2]},\qquad \text{for}\ x\ge a,
\end{align}
where $\mu$ is a measure on $[a,\infty)$. Taking the Laplace transform, for $\lambda>0$, the left-hand side of (\ref{SN-formally}) becomes
\begin{align*}
\int_a^\infty e^{-\lambda x}dx\int_{[a,x)} W(x-y)\mu(dy)&=\int_{[a,\infty)} \mu(dy)\int_y^\infty e^{-\lambda x}W(x-y)dx\\
&=\int_{[a,\infty)} e^{-\lambda y}\mu(dy)\int_0^{\infty }e^{-\lambda z}W(z)dz\\
&=\frac{1}{\psi(\lambda)}\int_{[a,\infty)} e^{-\lambda y}\mu(dy),
\end{align*}
where $\psi$ denotes the Laplace exponent of $X$. On the other hand, the Laplace transform of the right-hand side of (\ref{SN-formally}) is
\begin{align*}
\int_a^\infty e^{-\lambda x}\frac{x-a}{\E_0[X_1^2]}dx&=\frac{1}{\E_0[X_1^2]}\frac{e^{-\lambda a}}{\lambda^2}.
\end{align*}
Since $\E_0[X_1^2]<\infty$ and $X$ is recurrent, we have $\E_0[X_1]=0.$ Thus, we have
\begin{align*}
  \psi(\lambda)=\frac{1}{2}\sigma^2\lambda^2+\int_{(-\infty,0)}(e^{\lambda x}-1-\lambda x)\Pi(dx).
\end{align*}
By Fubini's theorem, we have
\begin{align*}
\psi(\lambda)&=\frac{1}{2}\sigma^2\lambda^2-\int_{(-\infty,0)}\Pi(dx)\int_x^0 (\lambda e^{\lambda y}-\lambda)dy\\
&=\frac{1}{2}\sigma^2\lambda^2-\int_{-\infty}^0 (\lambda e^{\lambda y}-\lambda)\overline{\Pi}(y)dy\\
&=\frac{1}{2}\sigma^2\lambda^2+\int_{-\infty}^0 \overline{\Pi}(y)dy\int_y^0 \lambda ^2 e^{\lambda z}dz\\
&=\frac{1}{2}\sigma^2 \lambda^2+\int_{-\infty}^0 \lambda^2 e^{\lambda z}\overline{\overline{\Pi}}(z)dz\\
&=\frac{1}{2}\sigma^2 \lambda^2+\int_0^\infty \lambda^2 e^{-\lambda z}\overline{\overline{\Pi}}(-z)dz.
\end{align*}
Therefore, we have
\begin{align*}
  \int_a^\infty e^{-\lambda y}\mu(dy)&=\frac{\psi(\lambda)}{\E_0[X_1^2]}\frac{e^{-\lambda a}}{\lambda^2}\\
  &=\frac{e^{-\lambda a}}{\E_0[X_1^2]}\left(\frac{\sigma^2}{2}+\int_0^\infty e^{-\lambda z}\overline{\overline{\Pi}}(-z)dz\right)\\
  &=\frac{\sigma^2e^{-\lambda a}}{2\E_0[X_1^2]}+\frac{1}{\E_0[X_1^2]}\int_a^\infty e^{-\lambda y}\overline{\overline{\Pi}}(a-y)dy.
\end{align*}
By the uniqueness of the Laplace transform, we have
\begin{align*}
  \mu(dy)=\frac{\sigma^2}{2\E_0[X_1^2]}\delta_a(dy)+\frac{1}{\E_0[X_1^2]}\overline{\overline{\Pi}}(a-y)1_{(a,\infty)}(y)dy.
\end{align*}
Since
\begin{align*}
  \mu([a,\infty))&=\frac{\sigma^2}{2\E_0[X_1^2]}+\frac{1}{\E_0[X_1^2]}\int_a^\infty \overline{\overline{\Pi}}(a-y)dy\\
  &=\frac{1}{\E_0[X_1^2]}\left(\frac{\sigma^2}{2}+\frac{1}{2}\int_{(-\infty,0)} x^2 \Pi(dx)\right)\\
  &=\frac{1}{2},
\end{align*}
the measure
\begin{align*}
  \mu(dy)1_{[a,b)}(y)+\delta_b(dy)\Big(1-\mu([a,b))\Big)
\end{align*}
is a probability measure supported on $[a,b]$ and satisfies (\ref{SN-equibrium}). Moreover, we have
\begin{align*}
1-\mu([a,b))&=\frac{1}{2}+\mu([b,\infty))\\
&=\frac{1}{2}+\frac{1}{\E_0[X_1^2]}\int_b^\infty \overline{\overline{\Pi}}(a-y)dy\\
&=\frac{1}{2}+\frac{1}{2\E_0[X_1^2]}\int_{(-\infty,a-b)}(a-b-x)^2 \Pi(dx).
\end{align*}
Thus, by the uniqueness of the equilibrium measure, we obtain (\ref{SN-equibri}). Moreover, by (\ref{SN-Robin}), we obtain the Robin constant (\ref{SN-Robinconst}). This completes the proof.
\end{proof}

%%%%%%%%%%%%%%%%%%%%%%%%%%%%%%%%%%%%%%%%%%%%%%%%%%%%%%%%%%%%%%%%%%%%%%%%%%%%%%%%%%%%%%%%%%%%%%%%%%%%%%%%%%%%%%%%%%%%%%%%%%%%%%%%%%%%%%%%%%%%%%%%%%%%%%%%%%%%%%%%%%%%%%%%%%%%%%%%%%%%%%%%%%%%%%%%%%%%%%%%%%%%%%%%%%%%%%%%%%%%%%%%%%%%%%%%%%%%%%%%%%%%%%%%%%%%%%%%%%%%%%%%%%%%%%%%%%%%%%%%%%%%%%%%%%%%%%%%

\section{Transient case}
\label{App}
For the sake of completeness, we deal with the transient case. We consider the following two conditions:
\begin{enumerate}
  \item[\textbf{(A1)}] The process $X$ is not a compound Poisson process.
  \item[\textbf{(A2)}] The point $0$ is regular for itself.
\end{enumerate}
It is known that condition \textbf{(\ref{ConditionA})} implies both \textbf{(A1)} and \textbf{(A2)}, see, e.g., page 7 of \cite{TY}. Throughout this section, we assume that $X$ is transient and satisfies conditions \textbf{(A1)} and \textbf{(A2)}. Under these assumptions, the local time $(L_t^a,t\ge 0)$ and the resolvent density $r_q(x)$ can be defined in the same way as in the recurrent case. However, in contrast to the recurrent case, the following limit is finite and strictly positive:
\begin{align*}
  \kappa:=\lim_{q\to 0}\frac{1}{r_q(0)}\in (0,\infty),
\end{align*}
see, e.g., Theorem I.17 of \cite{Ber}. Takeda-Yano \cite{TY} defined the renormalized zero resolvent $h$ in the transient case by
\begin{align*}
  h(x):=\lim_{q\to 0}(r_q(0)-r_q(-x))=\frac{1}{\kappa}\P_x(T_{\{0\}}=\infty),
\end{align*}
see Theorem 9.1 of \cite{TY}. Similarly, for a compact set $B$, we obtain
\begin{align}
\label{transient-varphi}
  \varphi_B(x):=\lim_{q\to 0}r_q(0)\E_x[1-e^{-qT_B}]=\frac{1}{\kappa}\P_x(T_B=\infty).
\end{align}

We next consider the equilibrium measure. For $q> 0$, the $q$-capacity measure $m_B^q$ of a compact set $B$ exists and satisfies
\begin{align*}
  \E_x[e^{-qT_B}]=\int r_q(y-x)m_B^q(dy),
\end{align*}
see, e.g., Proposition 42.13 of \cite{Sato}. Since the process $X$ is transient, there exist the zero resolvent density $r_0(x)$ and the zero-capacity measure $m_B^0(dy)$ such that
\begin{align}
\label{zero-capacity}
  \P_x(T_B<\infty)=\int r_0(y-x) m_B^0(dy),
\end{align}
see, e.g., Proposition 42.13 of \cite{Sato}. Moreover, $m_B^q$ converges weakly to $m_B^0$ and $r_q(x)\to r_0(x)$ for each $x$, see, e.g., Corollary II.8 of \cite{Ber}. This measure $m_B^0$ is also called the equilibrium measure. Note that $m_B^0$ is a finite measure, see, e.g., Theorem 42.8 of \cite{Sato}. By (\ref{zero-capacity}), we have
\begin{align*}
\varphi_B(x)&=\frac{1}{\kappa}-\frac{1}{\kappa}\P_{x}(T_B<\infty)\\
&=\frac{1}{\kappa}-\frac{1}{\kappa}\int r_0(y-x) m_B^0(dy)\\
&=\frac{1}{\kappa}+\frac{1}{\kappa}\int h(x-y)m_B^0(dy)-\frac{1}{\kappa}\int r_0(0)m_B^0(dy)\\
&=\frac{1}{\kappa}\int h(x-y)m_B^0(dy)-\frac{1}{\kappa}\left(\frac{C(B)}{\kappa}-1\right),
\end{align*}
where $C(B):=m_B^0(B)$ is the zero-capacity of $B$. This is the representation corresponding to (\ref{main1-eq}). Theorem \ref{main2} can also be proved in the same way in the transient case. In particular, by taking the Revuz measure $\nu$ to be $m_B^0$, we obtain the following Tanaka-type formula in the transient case:
\begin{align}
\label{Transient=Tanaka}
  \varphi_B(X_t)=\varphi_B(x)+\frac{1}{\kappa}M_t^{m_B^0}+\frac{1}{\kappa}\int_B L_t^y m_B^0(dy),
\end{align}
where $(M_t^{m_B^0})$ is a martingale. Multiplying both sides of (\ref{Transient=Tanaka}) by $\kappa$ and using (\ref{transient-varphi}), we can rewrite the formula in terms of hitting probabilities as follows:
\begin{align*}
  \P_{X_t}(T_B=\infty)=\P_x(T_B=\infty)+M_t^{m_B^0}+\int_B L_t^y m_B^0(dy),\qquad t\geq 0.
\end{align*}

%%%%%%%%%%%%%%%%%%%%%%%%%%%%%%%%%%%%%%%%%%%%%%%%%%%%%%%%%%%%%%%%%%%%%%%%%%%%%%%%%%%%%%%%%%%%%%%%%%%%%%%%%%%%%%%%%%%%%%%%%%%%%%%%%%%%%%%%%%%%%%%%%%%%%%%%%%%%%%%%%%%%%%%%%%%%%%%%%%%%%%%%%%%%%%%%%%%%%%%%%%%%%%%%%%%%%%%%%%%%%%%%%%%%%%%%%%%%%%%%%%%%%%%%%%%%%%%%%%%%%%%%%%%%%%%%%%%%%%%%%%%%%%%%%%%%%%%

\appendix
\section{Appendix: Proof of Lemmas in Subsection \ref{SS5.1}}
\label{App-lem}
In this section, we provide the proofs of the lemmas postponed above.

\begin{proof}[Proof of Lemma \ref{Stable-lemma1}]

First, suppose that $x<-1.$ By Lemma \ref{interval-lem} and the definition of $\mu$, we have
\begin{align*}
\label{int-compute1}
&\int_\R h(x-y)\mu(dy)=C_3 I(x),
  \stepcounter{equation}\tag{\theequation}
\end{align*}
where
\begin{align*}
  C_3&:=-\frac{\cos C_1\sin (\pi\alpha\rho)}{2^{1-\alpha}B(1-\alpha\hat{\rho},1-\alpha\rho)\Gamma(\alpha)\sin (\pi\alpha)},\\
  I(x)&:=\int_{-1}^1 (y-x)^{\alpha-1}(1-y)^{-\alpha\rho}(1+y)^{-\alpha \hat{\rho}}dy.
\end{align*}
Let us compute $I(x)$. Making the change of variable $s=\frac{(-x+1)t}{-x-1+2t}$ or equivalently $t=\frac{(-x-1)s}{-x+1-2s}$, we have
\begin{align*}
I'(x)&=-(\alpha-1)\int_{-1}^1 (y-x)^{\alpha-2}(1-y)^{-\alpha\rho}(1+y)^{-\alpha \hat{\rho}}dy\\
&=-(\alpha-1)2^{1-\alpha}\int_0^1 (2t-1-x)^{\alpha-2}(1-t)^{-\alpha\rho}t^{-\alpha \hat{\rho}}dt\qquad (y=2t-1)\\
&=-(\alpha-1)2^{1-\alpha}\int_0^1 \left(\frac{(1-x)(-x-1)}{-x+1-2s}\right)^{\alpha-2}\left(\frac{(1-x)(1-s)}{-x+1-2s}\right)^{-\alpha\rho}\\
&\qquad\qquad\qquad\qquad\qquad\qquad  \times \left(\frac{(-x-1)s}{-x+1-2s}\right)^{-\alpha \hat{\rho}}\frac{(-x-1)(1-x)}{(-x+1-2s)^2}ds\\
&=-(\alpha-1)2^{1-\alpha}(1-x)^{\alpha \hat{\rho}-1}(-x-1)^{\alpha \rho-1}\int_0^1 (1-s)^{-\alpha \rho}s^{-\alpha\hat{\rho}}ds\\
&=-(\alpha-1)2^{1-\alpha}(1-x)^{\alpha \hat{\rho}-1}(-x-1)^{\alpha \rho-1}B(1-\alpha \hat{\rho},1-\alpha\rho).
\end{align*}
Thus, we have
\begin{align*}
&I(-1-\e)-I(x)\\
&\qquad =\int_{x}^{-1-\e}I'(y)dy\\
&\qquad =-(\alpha-1)2^{1-\alpha}B(1-\alpha \hat{\rho},1-\alpha\rho)\int_x^{-1-\e}(1-y)^{\alpha \hat{\rho}-1}(-y-1)^{\alpha \rho-1}dy\\
&\qquad =-(\alpha-1)2^{1-\alpha}B(1-\alpha \hat{\rho},1-\alpha\rho)\int_{1+\e}^{|x|}(1+y)^{\alpha \hat{\rho}-1}(y-1)^{\alpha \rho-1}dy.
\end{align*}
Since
\begin{align*}
\lim_{\e\to 0}I(-1-\e)&=\lim_{\e\to 0}\int_{-1}^1 (y+1+\e)^{\alpha-1}(1-y)^{-\alpha\rho}(1+y)^{-\alpha \hat{\rho}}dy\\
&=\int_{-1}^1 (1-y)^{-\alpha \rho}(1+y)^{\alpha\rho-1}dy\\
&=B(\alpha\rho,1-\alpha \rho)\\
&=\frac{\pi}{\sin (\pi \alpha \rho)},
\end{align*}
we have
\begin{align*}
  I(x)&=C_4\int_{1}^{|x|}(1+y)^{\alpha \hat{\rho}-1}(y-1)^{\alpha \rho-1}dy+\frac{\pi}{\sin (\pi\alpha \rho)},
\end{align*}
where 
\begin{align*}
  C_4:=(\alpha-1)2^{1-\alpha}B(1-\alpha \hat{\rho},1-\alpha\rho).
\end{align*}
Hence, using (\ref{int-compute1}) and (\ref{phi-phi}), we obtain
\begin{align*}  
&\int_\R h(x-y)\mu(dy)\\
&\qquad =C_3C_4\int_{1}^{|x|}(1+y)^{\alpha \hat{\rho}-1}(y-1)^{\alpha \rho-1}dy+\frac{\pi C_3}{\sin(\pi\alpha\rho)}\\
&\qquad =C_3C_4\cdot \frac{1}{\cos C_1 \cdot K(\alpha,\rho)}\frac{\Gamma(\alpha\rho)}{\Gamma(1-\alpha\hat{\rho})}\cdot \varphi_{[-1,1]}(x)+\frac{\pi C_3}{\sin(\pi\alpha\rho)}.
\end{align*}
By the reflection formula for the Gamma function, we have
\begin{align*}
&C_3C_4\cdot \frac{1}{\cos C_1 \cdot K(\alpha,\rho)}\frac{\Gamma(\alpha\rho)}{\Gamma(1-\alpha\hat{\rho})}\\
&\qquad =-\frac{\cos C_1\sin (\pi\alpha\rho)}{2^{1-\alpha}B(1-\alpha\hat{\rho},1-\alpha\rho)\Gamma(\alpha)\sin (\pi\alpha)}\cdot (\alpha-1)2^{1-\alpha}B(1-\alpha \hat{\rho},1-\alpha\rho)\\
&\qquad \qquad \times \frac{\Gamma(1-\alpha \rho)\Gamma(1-\alpha \hat{\rho})}{\cos C_1\Gamma(2-\alpha)}\frac{\Gamma(\alpha\rho)}{\Gamma(1-\alpha\hat{\rho})}\\
&\qquad =\frac{\sin (\pi\alpha\rho)}{\sin (\pi\alpha)}\cdot \frac{\Gamma(\alpha\rho)\Gamma(1-\alpha \rho)}{\Gamma(\alpha)\Gamma(1-\alpha)}\\
&\qquad =\frac{\sin (\pi\alpha\rho)}{\sin (\pi\alpha)}\cdot \frac{\frac{\pi}{\sin (\pi \alpha\rho)}}{\frac{\pi}{\sin(\pi\alpha)}}\\
&\qquad =1.
\end{align*}
Therefore, we obtain (\ref{Stable-lemma1-eq}) for $x<-1$. The case $x>1$ is treated similarly. This completes the proof.
\end{proof}

\begin{proof}[Proof of Lemma \ref{Stable-lemma2}]
For $x\in (-1,1)$, we have
\begin{align*}
\int_\R h(x-y)\mu(dy)&=\hat{C}_3I_1(x)+C_3 I_2(x),
\end{align*}
where
\begin{align*}
  \hat{C}_3&:=-\frac{\cos C_1\sin (\pi\alpha\hat{\rho})}{2^{1-\alpha}B(1-\alpha\hat{\rho},1-\alpha\rho)\Gamma(\alpha)\sin (\pi\alpha)},\\
  I_1(x)&:=\int_{-1}^x (x-y)^{\alpha-1}(1-y)^{-\alpha\rho}(1+y)^{-\alpha \hat{\rho}}dy,\\
  I_2(x)&:=\int_x^1 (y-x)^{\alpha-1}(1-y)^{-\alpha \rho}(1+y)^{-\alpha\hat{\rho}}dy.
\end{align*}
It suffices to prove that the derivative of this function is zero. By Leibniz's rule and making the change of variable $y=\frac{(1+x)u-(1-x)}{(1+x)u+(1-x)}$, we have
\begin{align*}
I_1'(x)&=(\alpha-1)\int_{-1}^x (x-y)^{\alpha-2}(1-y)^{-\alpha\rho}(1+y)^{-\alpha \hat{\rho}}dy\\
&=(\alpha-1)\int_{0}^1\left(\frac{(1-x^2)(1-u)}{(1+x)u+(1-x)}\right)^{\alpha-2}\left(\frac{2(1-x)}{(1+x)u+(1-x)}\right)^{-\alpha \rho}\\
&\qquad \qquad \qquad \qquad \times \left(\frac{2(1+x)u}{(1+x)u+(1-x)}\right)^{-\alpha \hat{\rho}}\frac{2(1-x^2)}{((1+x)u+(1-x))^2}du\\
&=(\alpha-1)2^{1-\alpha}(1-x)^{\alpha \hat{\rho}-1}(1+x)^{\alpha \rho -1}\int_0^1 (1-u)^{\alpha-2}u^{-\alpha \hat{\rho}}du\\
&=(\alpha-1)2^{1-\alpha}(1-x)^{\alpha \hat{\rho}-1}(1+x)^{\alpha \rho -1}B(1-\alpha \hat{\rho},\alpha-1).
\end{align*}
Similarly, we have
\begin{align*}
I_2'(x)=-(\alpha-1)2^{1-\alpha}(1-x)^{\alpha \hat{\rho}-1}(1+x)^{\alpha \rho -1}B(\alpha-1,1-\alpha\rho).
\end{align*}
By the reflection formula for the Gamma function, we have
\begin{align*}
&\sin (\pi \alpha\hat{\rho})B(1-\alpha \hat{\rho},\alpha-1)-\sin (\pi\alpha\rho)B(\alpha-1,1-\alpha{\rho})\\
&\qquad =\sin (\pi \alpha\hat{\rho})\cdot \frac{\Gamma(\alpha-1)\Gamma(1-\alpha\hat{\rho})}{\Gamma(\alpha\rho)}-\sin (\pi\alpha\rho)\frac{\Gamma(\alpha-1)\Gamma(1-\alpha{\rho})}{\Gamma(\alpha\hat{\rho})}\\
&\qquad =\frac{\pi}{\Gamma(\alpha\hat{\rho})}\cdot \frac{\Gamma(\alpha-1)}{\Gamma(\alpha\rho)}-\frac{\pi}{\Gamma(\alpha\rho)}\cdot\frac{\Gamma(\alpha-1)}{\Gamma(\alpha\hat{\rho})}\\
&\qquad =0.
\end{align*}
Hence, we have
\begin{align*}
  \hat{C}_3I_1'(x)+C_3I_2'(x)=0\qquad \text{for}\ x\in (-1,1).
\end{align*}
Therefore, the function $x\mapsto\int_\R h(x-y)\mu(dy)$ is constant on $[-1,1]$. This completes the proof.
\end{proof}

%%%%%%%%%%%%%%%%%%%%%%%%%%%%%%%%%%%%%%%%%%%%%%%%%%%%%%%%%%%%%%%%%%%%%%%%%%%%%%%%%%%%%%%%%%%%%%%%%%%%%%%%%%%%%%%%%%%%%%%%%%%%%%%%%%%%%%%%%%%%%%%%%%%%%%%%%%%%%%%%%%%%%%%%%%%%%%%%%%%%%%%%%%%%%%%%%%%%%%%%%%%%%%%%%%%%%%%%%%%%%%%%%%%%%%%%%%%%%%%%%%%%%%%%%%%%%%%%%%%%%%%%%%%%%%%%%%%%%%%%%%%%%%%%%%%%%%%

\section*{Acknowledgements}
This research benefited greatly from insightful comments from Prof. Kouji Yano; the authors would like to thank him for this. The first author acknowledges support from JSPS KAKENHI Grant Number 26KJ1605. Part of this work was conducted during a visit of the first author to CIMAT. Both authors would like to thank CIMAT and the JSPS Open Partnership Joint Research Projects, Grant Number JPJSBP120209921 for the funding provided for this.

\bibliographystyle{plain}
%\bibliography{reference}

\begin{thebibliography}{10}

\bibitem{App}
D.~Applebaum.
\newblock {\em L\'evy processes and stochastic calculus}, volume 116 of {\em Cambridge Studies in Advanced Mathematics}.
\newblock Cambridge University Press, Cambridge, second edition, 2009.

\bibitem{Ber}
J.~Bertoin.
\newblock {\em L\'{e}vy processes}, volume 121 of {\em Cambridge Tracts in Mathematics}.
\newblock Cambridge University Press, Cambridge, 1996.

\bibitem{Boy}
E.~S. Boylan.
\newblock Local times for a class of {M}arkoff processes.
\newblock {\em Illinois J. Math.}, 8:19--39, 1964.

\bibitem{CD}
L.~Chaumont and R.~A. Doney.
\newblock On {L}\'{e}vy processes conditioned to stay positive.
\newblock {\em Electron. J. Probab.}, 10:no. 28, 948--961, 2005.

\bibitem{CW}
K.~L. Chung and R.~J. Williams.
\newblock {\em Introduction to stochastic integration}.
\newblock Modern Birkh\"auser Classics. Birkh\"auser/Springer, New York, second edition, 2014.

\bibitem{DMVolV-VII}
C.~Dellacherie, and P.A. Meyer.
\newblock {\em Probabilit\'es et Potentiel. Chapitres V {\'a} VII. Theorie des martingales}
\newblock Hermann, Paris, second edition, 1980.

\bibitem{DKW}
L.~D\"oring, A.~E. Kyprianou, and P.~Weissmann.
\newblock Stable processes conditioned to avoid an interval.
\newblock {\em Stochastic Process. Appl.}, 130(2):471--487, 2020.

\bibitem{EisenbaumWalsh}
\newblock N.~Eisenbaum, A.~Walsh.
\newblock An optimal Itô formula for \Levy\ processes. 
\newblock {\em Electron. Commun. Probab.} 14 202 - 209, 2009. 

\bibitem{fukushima}
M.~Fukushima. 
\newblock {\em A decomposition of additive functionals of finite energy.} 
\newblock Nagoya Math. J. 74 137–168, 1979.

\bibitem{fukushimaetal}
M.~Fukushima, and Y.~Oshima, and M.~Takeda, M. 
\newblock {\em Dirichlet Forms and Symmetric Markov Processes.}
\newblock de Gruyter Studies in Mathematics 19. Berlin: de Gruyter. 1994.

\bibitem{G}
R.~K. Getoor.
\newblock Continuous additive functionals of a {M}arkov process with applications to processes with independent increments.
\newblock {\em J. Math. Anal. Appl.}, 13:132--153, 1966.


\bibitem{Iba}
K.~Iba.
\newblock Conditioning to avoid bounded sets for a one-dimensional {L}\'evy processes, preprint, arXiv: 2501.02776.

\bibitem{Iba-Qmat}
K.~Iba.
\newblock Hitting probabilities of finite points for one-dimensional {L}\'{e}vy processes, preprint, arXiv: 2602.09342.

\bibitem{IY-3}
K.~Iba and K.~Yano.
\newblock Two-point local time penalizations with various clocks for {L}\'evy processes.
\newblock {\em ALEA Lat. Am. J. Probab. Math. Stat.}, 22(1):183--207, 2025.

\bibitem{Kaspi}
H.~Kaspi.
\newblock Excursions of {M}arkov processes: an approach via {M}arkov additive processes.
\newblock {\em Z. Wahrsch. Verw. Gebiete}, 64(2):251--268, 1983.

\bibitem{Klenke}
A.~Klenke.
\newblock {\em Probability theory---a comprehensive course}.
\newblock Universitext. Springer, Cham, third edition, [2020] \copyright 2020.

\bibitem{Kyp}
A.~E. Kyprianou.
\newblock {\em Fluctuations of {L}\'{e}vy processes with applications}.
\newblock Universitext. Springer, Heidelberg, second edition, 2014.
\newblock Introductory lectures.

\bibitem{Pa}
H.~Pant\'{\i}.
\newblock On {L}\'{e}vy processes conditioned to avoid zero.
\newblock {\em ALEA Lat. Am. J. Probab. Math. Stat.}, 14(2):657--690, 2017.

\bibitem{PStable}
S.~C. Port.
\newblock Hitting times and potentials for recurrent stable processes.
\newblock {\em J. Analyse Math.}, 20:371--395, 1967.

\bibitem{SY}
P.~Salminen and M.~Yor.
\newblock Tanaka formula for symmetric {L}\'evy processes.
\newblock In {\em S\'eminaire de {P}robabilit\'es {XL}}, volume 1899 of {\em Lecture Notes in Math.}, pages 265--285. Springer, Berlin, 2007.

\bibitem{Sato}
K.~Sato.
\newblock {\em L\'{e}vy processes and infinitely divisible distributions}, volume~68 of {\em Cambridge Studies in Advanced Mathematics}.
\newblock Cambridge University Press, Cambridge, 1999.
\newblock Translated from the 1990 Japanese original, Revised by the author.

\bibitem{Sharpe}
M.~Sharpe.
\newblock {\em General theory of {M}arkov processes}, volume 133 of {\em Pure and Applied Mathematics}.
\newblock Academic Press, Inc., Boston, MA, 1988.

\bibitem{Takeda}
S.~Takeda.
\newblock Sample path behaviors of {L}\'evy processes conditioned to avoid zero.
\newblock {\em J. Theoret. Probab.}, 37(4):3177--3201, 2024.

\bibitem{TY}
S.~Takeda and K.~Yano.
\newblock Local time penalizations with various clocks for {L}\'{e}vy processes.
\newblock {\em Electron. J. Probab.}, 28:Paper No. 12, 35, 2023.

\bibitem{Tukada}
H.~Tsukada.
\newblock A potential theoretic approach to {T}anaka formula for asymmetric {L}\'{e}vy processes.
\newblock In {\em S\'{e}minaire de {P}robabilit\'{e}s {XLIX}}, volume 2215 of {\em Lecture Notes in Math.}, pages 521--542. Springer, Cham, 2018.

\bibitem{T2}
H.~Tsukada.
\newblock Tanaka formula for strictly stable processes.
\newblock {\em Probab. Math. Statist.}, 39(1):39--60, 2019.

\bibitem{Walsh1}
A.~Walsh. 
\newblock Stochastic integration with respect to additive functionals of zero quadratic variation. 
\newblock {\em Bernoulli} 19 (5B) 2414 - 2436, 2013. 


\bibitem{Wink}
M.~Winkel.
\newblock Right inverses of nonsymmetric {L}\'evy processes.
\newblock {\em Ann. Probab.}, 30(1):382--415, 2002.

\bibitem{Yamada}
K.~Yamada.
\newblock Fractional derivatives of local times of {$\alpha$}-stable {L}evy processes as the limits of occupation time problems.
\newblock In {\em Limit theorems in probability and statistics, {V}ol. {II} ({B}alatonlelle, 1999)}, pages 553--573. J\'anos Bolyai Math. Soc., Budapest, 2002.

\bibitem{Yano3}
K.~Yano.
\newblock Excursions away from a regular point for one-dimensional symmetric {L}\'evy processes without {G}aussian part.
\newblock {\em Potential Anal.}, 32(4):305--341, 2010.

\bibitem{Yano2}
K.~Yano.
\newblock On harmonic function for the killed process upon hitting zero of asymmetric {L}\'{e}vy processes.
\newblock {\em J. Math-for-Ind.}, 5A:17--24, 2013.


\end{thebibliography}

\end{document}